\documentclass[a4paper,reqno,11pt]{amsart}

\usepackage{epsf,graphicx,epsfig}
\usepackage{amscd}
\usepackage{amsmath,latexsym,amssymb,amsthm}
\usepackage[nospace,noadjust]{cite}
\usepackage{textcomp}
\usepackage{setspace,cite}
\usepackage{lscape,fancyhdr,fancybox}
\usepackage[all,cmtip]{xy}
\usepackage{tikz}
\usepackage{tikz-cd}
\usetikzlibrary{shapes,arrows,decorations.markings}
\usepackage{graphicx}

\usepackage{color}
\usepackage{url}
\usepackage{enumerate}
\usepackage[mathscr]{euscript}

  \theoremstyle{plain}

\swapnumbers
    \newtheorem{Thm}{Theorem}[section]
    \newtheorem{Prop}[Thm]{Proposition}
   \newtheorem{Lem}[Thm]{Lemma}
    \newtheorem{Cor}[Thm]{Corollary}
    
    \newtheorem{subsec}[Thm]{}
\theoremstyle{definition}
    \newtheorem{Def}[Thm]{Definition}
    \newtheorem{Exm}[Thm]{Example}

\theoremstyle{remark}
     \newtheorem{Rem}[Thm]{Remark}

\title{}
\author{}
\date{}
\usepackage{amssymb}

\usepackage{hyperref}
\hypersetup{
	colorlinks,
	citecolor=blue,
	filecolor=black,
	linkcolor=blue,
	urlcolor=black
}

\begin{document}
\title{Deformation Cohomology of Lie-Yamaguti algebras and Free Lie-Yamaguti algebras}

\author{Saikat Goswami}
\email{saikatgoswami.math@gmail.com}
\address{TCG CREST, Institute for Advancing Intelligence, Kolkata, India.}
\address{RKMVERI, Belur, Howrah, India.}


\thispagestyle{empty}

\begin{abstract}
Infinitesimal deformation theory of Lie-Yamaguti algebras was introduced by Tao Zhang and Juan Li . We extend their theory to develop formal one-parameter deformation theory of Lie-Yamaguti algebras. It turns out that the right deformation cohomology which controls deformations in this context, is a variant of Lie-Yamaguti algebra cohomology which we introduce in this paper. At the end, we introduce the free object in the category of Lie-Yamaguti algebras and prove that it is rigid.
\end{abstract}
\maketitle


\vspace{-0.8cm}
\section{\Large Introduction}\label{$1$} 
Deformation theory originated from Riemann's 1857  memoir on abelian functions in which he studied complex manifolds of dimension $1$ (Riemann surfaces) and calculated the number of parameters (called modulii) upon which a deformation depends. His idea was later extensively developed to study deformations of structures on manifolds by Frolicher-Kodaira-Nijenhuis-Nirenberg-Spencer (\cite{KS1}, \cite{KS2}, \cite{KNS}, \cite{Spencer1}, \cite{Spencer2}). 

The study of deformations of algebraic structures was initiated by M. Gerstenhaber (\cite{MG1}, \cite{MG2}, \cite{MG3}, \cite{MG4}, \cite{MG5}) which laid the foundation of Algebraic Deformation Theory. His theory was extended to Lie algebras by A. Nijenhuis and R. Richardson (\cite{NR1}, \cite{NR2}, \cite{NR3}).

J. L. Loday introduced non anti-symmetric analogue of Lie algebras, known as Leibniz algebras and some more binary quadratic algebras (with more than one binary operation) satisfying some intriguing relations, generally known as Loday algebras in \cite{Loday1}, \cite{Loday2 } and introduced related operads.

Algebraic Deformation theory of Dialgebras which involve two binary operations was developed in \cite{MM1}, \cite{MM2}. All the deformation theory mentioned above are formal one parameter deformations.

Algebraic deformation theory for Leibniz algebras over a commutative local algebra base was developed in \cite{FMM} and the versal object was constructed. 

In \cite{FMN}, the authors developed  Versal deformation theory of algebras over a quadratic operad.

M. Gerstenhaber in his paper remarked that his method would extend to any equation-ally defined algebraic structure. Therefore, it is still an active area of research to develop deformation theory of a given algebraic structure defined on a vector space.

\begin{Rem}\label{feature}
Any deformation theory should have the following features:
\noindent

\begin{enumerate}
\item A class of objects within which deformation takes place and a cohomology theory (deformation cohomology) associated to those objects which controls the deformation in the sense that infinitesimal deformation of a given object can be identified with the elements of a suitable cohomology group.

\item A theory of obstruction to the integration of an infinitesimal deformation.

\item A notion of equivalence of deformed objects and a notion of rigid objects.
\end{enumerate}
\end{Rem}
The notion of  Lie triple system was formally introduced as an algebraic object by Jacobson \cite{NJ} in connection with problems which arose from quantum mechanics (cf. \cite{Duffin}). Nomizu \cite{KN} proved that affine connections
with parallel torsion and curvature are locally equivalent to invariant connections
on reductive homogeneous spaces, and that each such space has a canonical
connection for which parallel translation along geodesics agrees with the natural
action of the group. K. Yamaguti \cite{KY} introduced the notion of {\it general Lie triple system} as a vector space equipped with a pair consisting of a bilinear and a trilinear operation and satisfying some intriguing relations  in order to characterize the torsion and curvature tensors of Nomizu's canonical connection. M. Kikkawa \cite{MK75} observed that the problem is related to the canonical connection and to the general Lie triple system defined on the tangent space $T_eM.$ In this paper, Kikkawa renamed the notion of general Lie system as {\it Lie triple algebra}. Kinyon and Weinstein \cite{KW} observed that Lie triple algebras, which they called Lie-Yamaguti algebras in their paper, can be constructed from Leibniz algebras. 

In \cite{ZT}, the author studied infinitesimal deformations of a Lie triple system using the cohomology groups as introduced by K. Yamaguti in \cite{KY-LTS}. Later, in \cite{ZL}, the authors studied infinitesimal deformations and extensions of Lie-Yamaguti algebras. They proved that infinitesimal deformations and abelian extensions of a Lie-Yamaguti algebra can be characterized by the $(2,3)$-cohomology group of the Lie-Yamaguti algebra (\cite{KY}, \cite{KY-cohomology}). Because of the complicated nature of the cohomology of Lie-Yamaguti algebras, till date it is not clear as to how one should go about developing an obstruction theory to integrate an infinitesimal deformation of a Lie-Yamaguti algebra to a full-blown formal one parameter deformation.

An attempt has been made in \cite{LCM} to develop a formal one parameter deformation theory of Lie-Yamaguti algebras. Unfortunately, there are serious mistakes in the paper, particularly, in Section $6.$ In Lemma $6.2,$ it has been assumed that the coefficients in a deformation of order $(n-1)$ are cocycles. This is an awkward and unreasonable assumption. In a finite order deformation, only the infinitesimal is a cocycle. The proof of the lemma is based on this assumption, and hence, obstruction theory to integrate a finite order deformation to a full-blown deformation has not been developed correctly. 

The aim of this paper, is to correct this mistake in the literature and develop a formal deformation theory extending the work of Tao Zhang and Juan Li \cite{ZL} of Lie-Yamaguti algebras following the features listed in Remark \ref{feature} mentioned above. The cohomology as defined by Yamaguti cannot be used directly as the deformation cohomology unless the cochain complex defining the cohomology is modified suitably. To do this, we notice that all the necessary tools are already in-built in the Yamaguti's definition. We modify it suitably to get the deformation cohomology which controls deformations in the present context. 

\medskip
\noindent
{\bf Organization of the paper}: In \S \ref{$2$}, we set up notations, recall some known definitions and results. In \S \ref{$3$}, we introduce the notion of formal one-parameter deformation of a Lie-Yamaguti algebra over a field $\mathbb K$ of characteristic $0.$ Next, we introduce the notion of an infinitesimal of a deformation, and show that infinitesimal of a deformation is a $(2, 3)$-cocycle. In \S \ref{$4$},  we address the question of integrability of a given $(2,3)$-cocycle in $Z^{(2, 3)}(L; L).$ We show that it is necessary to modify the cochain complex of Yamaguti as introduced in Section \ref{$2$} to obtain the deformation cohomology that controls deformations of Lie-Yamaguti algebras. In order to extend a deformation of order $n$ to a deformation of order $n+1$, there are obstructions. We identify the obstruction as a suitable cocycle of the deformation complex. We show that there exists a sequence of cohomology classes, vanishing of which is a necessary and sufficient condition for a given $(2,3)$-cocycle to be integrable. In \S \ref{$5$}, we introduce the notion of equivalence of deformations and rigid object. At the end, in \S \ref{$6$} we introduce the free object in the category of Lie-Yamaguti algebras and prove that it is rigid.

\vspace{0.5cm}
\section{\Large Preliminaries}\label{$2$}
Throughout this paper, we work on an algebraically closed field $\mathbb K$ of characteristic $0$.
We begin by recalling some basic definitions and set up our notations.
\begin{Def}
A Lie algebra is a vector space $\mathfrak g$ over $\mathbb K$ equipped with a $\mathbb K$-bilinear operation $[~,~] : \mathfrak g\times \mathfrak g \rightarrow \mathfrak g$ satisfying
\begin{enumerate}
\item (Anti-symmetry): $[x, y] = -[y, x]$ for all $x, y \in \mathfrak g$;
\item (Jacobi identity): $[[x,y],z] + [[y,z],x] + [[z,x],y] =0$  for all $x, y, z \in \mathfrak g.$ 
\end{enumerate}
\end{Def}

\begin{Def}
A Leibniz algebra is a vector space $\mathfrak g$ over $\mathbb K$ equipped with a $\mathbb K$-bilinear operation $\cdot : \mathfrak g \times \mathfrak g \rightarrow \mathfrak g$ satisfying the Leibniz identity
$$ x\cdot (y\cdot z) = (x\cdot y)\cdot z + y\cdot (x\cdot z) \qquad \text{for all}\quad x, y, z \in \mathfrak g.$$
\end{Def}
It is easy to see that in presence of the anti-symmetric condition the Leibniz identity reduces to Jacobi identity. Thus, Lie algebras are examples of Leibniz algebras. See \cite{Loday1} for many other non-trivial examples of Leibniz algebras.

\begin{Def}
A Lie triple system is a vector space $\mathfrak g$ over $\mathbb K$ equipped with a $\mathbb K$-trilinear operation 
$$\{~,~,~\} : \mathfrak g\times \mathfrak g\times \mathfrak g \rightarrow \mathfrak g$$
satisfying the following three identities
\begin{enumerate}
\item $\{ x, y, z\} = -\{ y, x, z\}$ for all $x, y, z \in \mathfrak g$;
\item $\{ x,y,z\} + \{ y,z,x\} + \{ z,x,y\} = 0$ for all $x, y, z \in \mathfrak g$;
\item $\{ x, y, \{ u, v, w\} \} = \{ \{ x, y, u \} , v, w\} + \{ u, \{ x, y, v\} , w\} + \{ u, v, \{ x, y, w\} \} $\\ for all $x, y, u, v, w \in \mathfrak g.$
\end{enumerate}
\end{Def}

The following is an interesting example of a Lie triple system which arose from Physics \cite{NJ}.

\begin{Exm}\label{Meson}
We denote by $M_{n+1}(\mathbb R)$, the set of all $(n+1) \times (n+1)$ matrices over the field $\mathbb R$, which is an associative algebra with respect to matrix multiplication.  Let $\delta_{ij}$ denote the Kronecker delta symbol
\[ \delta_{ij}= \left\{\begin{array}{ll}
0 & i \neq j \\
1 & i=j 
\end{array}
\right.\]
and $e_{i,j}$ denote the elementary matrix which has $1$ in the $(i,j)$-entry as its only non-zero entry.
Let $\mathfrak g$ be the subspace of $M_{n+1}(\mathbb R)$ spanned by the matrices $G_i$ for $i=1,2,\cdots,n$, where $G_i = e_{i,n+1}-e_{n+1, i}.$ 
\vspace{0.1in}
As an example, for $n=3,$ the matrix $G_2 \in M_4(\mathbb R)$ is given by
\[
G_2 =
\begin{pmatrix}
0 & 0 & 0 & 0 \\
0 & 0 & 0 & 1 \\
0 & 0 & 0 & 0 \\
0 & -1 & 0 & 0 \\
\end{pmatrix}.
\]
\noindent
Then, the subspace $\mathfrak g$ is closed under the ternary product 
$$\{A, B, C\}:=[[A,B],C], ~~A, B, C \in \mathfrak g$$ where $[A,B] := AB - BA$ is the commutator bracket. Explicitly, the trilinear product of the basis elements are given by
\[
	[[G_i,G_j],G_k] = \delta_{ki} G_j - \delta_{kj} G_i.
\] 
It turns out that $(\mathfrak g, \{~, ~,~\})$ is a Lie triple system, first used in \cite{Duffin} to provide a significant and elegant algebraic formalism of Meson equations and was introduced formally as a Lie triple system by Jacobson \cite{NJ} and hence known as Meson field.
\end{Exm}
\begin{Rem}
Note that any Lie algebra $(\mathfrak g, [~,~])$ can be viewed as a Lie triple system with the trilinear operation
$$\{ x, y, z\} := [[x, y], z]$$  for all $x, y, z \in \mathfrak g.$
\end{Rem}
Next, we recall the notion of Lie-Yamaguti algebra and its cohomology.

\vspace{0.1in}
\noindent
{\bf Lie-Yamaguti Algebra:}
\begin{Def}
A Lie-Yamaguti algebra $(L, [~,~], \{~, ~, ~\})$ is a vector space $L$ equipped with a $\mathbb K$-bilinear operation 
$$[~, ~]: L\times L \rightarrow L$$ and a $\mathbb K$-trilinear operation 
$$\{~, ~, ~\} : L\times L\times L \rightarrow L$$ such that for all $x, y, z, u, v, w \in L$ the following relations hold:
\begin{equation}
  [x, y] = -[y,x]; \tag{LY1} \label{LY1} 
\end{equation}
\begin{equation}
 \{x, y, z\} = -\{y,x, z\}; \tag{LY2} \label{LY2}
\end{equation}
\begin{equation}
\Sigma_{\circlearrowleft (x, y, z)}([[x,y], z] + \{x, y, z\}) =0;  \tag{LY3} \label{LY3}
\end{equation}
\begin{equation}
\Sigma_{\circlearrowleft (x, y, z)}\{[x, y], z, u\} = 0;  \tag{LY4} \label{LY4}
\end{equation}
\begin{equation}
\{x, y, [u, v]\} = [\{x, y, u\}, v] + [u, \{x, y, v\}];  \tag{LY5} \label{LY5}
\end{equation}
\begin{equation}
\{x, y, \{u, v, w\}\}\\
= \{\{x, y, u\}, v, w\} + \{u, \{x, y, v\}, w\} + \{u, v, \{x, y, w\}\}.  \tag{LY6} \label{LY6} 
\end{equation}
Here, $\Sigma_{\circlearrowleft (x,y,z)}$ denotes the sum over cyclic permutations of x, y, and z.
\end{Def}

\begin{Rem}
Notice that if the trilinear product in a Lie-Yamaguti algebra is trivial, that is, if $\{~,~,~\} = 0,$ then (LY$2$), (LY$4$), (LY$5$), and (LY$6$) are trivial, and (LY$1$) and
(LY$3$) define a Lie algebra structure on $L$. On the other hand, if the binary product is trivial, that is, $[~,~] =0,$ then (LY$1$), (LY$4$), and (LY$5$) are trivial, and (LY$2$), (LY$3$), together with (LY$6$) define a Lie triple system on $L.$
\end{Rem}

The following result is well-known. 

\begin{Lem}\label{Lie to LYA}
Let $(\mathfrak g, [~, ~])$ be a Lie algebra over $\mathbb K$. Then, $\mathfrak g$ has a Lie-Yamaguti algebra structure induced by the given Lie bracket, the trilinear operation being:
$$\{a, b, c\} = [[a, b], c]$$ for all $a, b, c \in \mathfrak g$.
\end{Lem}
\begin{Exm}
Let $(\mathfrak g, \cdot)$ be a Leibniz algebra. Define a bilinear operation and a trilinear operation as follows:
$$[~, ~]: \mathfrak g \times \mathfrak g \to \mathfrak g, \quad [a, b] := a\cdot b-b\cdot a,~~ a, b \in \mathfrak g;$$
$$\{~, ~, ~\}:  \mathfrak g \times \mathfrak g \times \mathfrak g \to \mathfrak g, \quad \{a, b, c\} := -(a\cdot b)\cdot c,~~a, b, c \in \mathfrak g.$$ Then, $(\mathfrak g, [~,~],\{~,~,~\})$ is a Lie-Yamaguti algebra. 
\end{Exm}
Let $(\mathfrak g, \langle~,~\rangle)$ be a Lie algebra. Recall that a reductive decomposition of
$\mathfrak g$ is a vector space direct sum $\mathfrak g = \mathfrak h \oplus\mathfrak m$ satisfying $\langle\mathfrak h, \mathfrak h\rangle \subseteq\mathfrak h$ and $\langle\mathfrak h, \mathfrak m\rangle \subseteq \mathfrak m.$ In this case, we call $(\mathfrak h, \mathfrak m)$ a reductive pair.

\begin{Exm}\label{reductive}
Let $(\mathfrak g, \langle~, ~\rangle)$ be a Lie algebra with a reductive decomposition $\mathfrak g = \mathfrak h \oplus\mathfrak  m.$ Then, there exist a natural binary and ternary products on $\mathfrak m$ defined by
$$ [a, b] := \pi_{\mathfrak m}(\langle a, b\rangle),\qquad\{a, b, c\} := \langle \pi_{\mathfrak h}(\langle a, b\rangle), c \rangle,$$ where $\pi_{\mathfrak m}$ and $\pi_{\mathfrak h}$ are the projections on $\mathfrak m$ and $\mathfrak h$, respectively. These products endow $\mathfrak m$ with the structure of a Lie-Yamaguti algebra \cite{BBM}.
\end{Exm}

\begin{Exm}
Consider the vector space $L$ over $\mathbb K$ generated by $\{e_1, e_2, e_3\}.$ Define a bilinear operation $[~,~]$ and a trilinear operation $\{~, ~, ~\}$ on $L$ as follows.
$$ [e_1, e_2] = e_3,\qquad \{e_1, e_2, e_1\} = e_3.$$
All other brackets of the basis elements are either determined by the above definition or else zero. Then, $L$ with the above operations is a Lie-Yamaguti algebra.
\end{Exm}

See \cite{ABCO} for classification of some low dimensional nilpotent Lie-Yamaguti algebras.
  
\begin{Def}
Let $(L, [~,~], \{~,~,~\})$,  $(L^\prime, [~,~]^\prime, \{~,~,~\}^\prime)$ be two Lie-Yamaguti algebras. A homomorphism
$$ \phi :  (L, [~,~], \{~,~,~\}) \rightarrow (L^\prime, [~,~]^\prime, \{~,~,~\}^\prime)$$ of Lie-Yamaguti algebras is a $\mathbb K$-linear map $\phi :  L \rightarrow L^\prime$ satisfying
$$\phi ([x, y]) = [\phi (x), \phi (y)]^\prime,\qquad \phi (\{x, y, z\}) = \{\phi (x), \phi (y), \phi (z)\}^\prime$$ for all $x, y, z \in L.$
A homomorphism
$$ \phi :  (L, [~,~], \{~,~,~\}) \rightarrow (L^\prime, [~,~]^\prime, \{~,~,~\}^\prime)$$ of Lie-Yamaguti algebras is an isomorphism if there exists a homomorphism
$$\phi^\prime:  (L^\prime, [~,~]^\prime, \{~,~,~\}^\prime) \rightarrow (L, [~,~], \{~,~,~\})$$ such that $\phi^\prime \circ \phi = id_L$ and $\phi \circ \phi^\prime = id_{L^\prime}.$  The set of all self-isomorphisms of a Lie-Yamaguti algebra  $(L, [~,~], \{~,~,~\})$ is obviously a group under composition of maps and is denoted by ${Aut}_{LY}(L).$
\end{Def}

\noindent
{\bf Representation of a Lie-Yamaguti algebra:}
Let $L$ be a Lie-Yamaguti algebra over a field $\mathbb{K}$. We first recall the definition of representation of $L$ from \cite{KY}, \cite{KY-cohomology}. 

\begin{Def}\label{Representation of LYA}
Let $V$ be a vector space over $\mathbb{K}.$ Let $E(V)$ denote the Lie algebra of linear endomorphisms of $V$ with respect to the commutator bracket. A representation of $L$ on $V$ consists of a linear map $\rho:L \to E(V)$, and bilinear maps $D,\theta:L \times L \to E(V)$ such that the following relations hold:   
\begin{align}
&D(a,b) + \theta(a,b) - \theta(b,a) = [\rho(a),\rho(b)] - \rho([a,b]) \label{R1} \tag{R1}
\\
&\theta(a,[b,c]) - \rho(b) \theta(a,c) + \rho(c) \theta(a,b) = 0 \label{R2} \tag{R2} \\
&\theta([a,b],c) - \theta(a,c) \rho(b) + \theta(b,c) \rho(a) = 0
\label{R3} \tag{R3} \\
&\theta(c,d) \theta(a,b) - \theta(b,d) \theta(a,c) - \theta(a,\{b,c,d\}) + D(b,c) \theta(a,d) = 0 
\label{R4} \tag{R4} \\
&[D(a,b),\rho(c)] = \rho(\{a,b,c\})
\label{R5} \tag{R5} \\
&[D(a,b),\theta(c,d)] = \theta(\{a,b,c\},d) + \theta(c,\{a,b,d\}) \label{R6} \tag{R6}
\end{align}
\end{Def}
A representation as above will be denoted by $(V; \rho, D, \theta).$
\noindent
Taking $V=L$, there is a natural representation of $L$ on itself, called the adjoint representation of $L$, where the representation maps $\rho, D, \theta$ are given by 
\[\rho(a)(b) := [a,b], \quad D(a,b)(c) := \{a,b,c\}, \quad \theta(a,b)(c) :=  \{c,a,b\}.\]
 \noindent

The following additional relations can be deduced using relation (1), (2), (3), and (5) stated above. 
$$D([a,b],c) + D([b,c],a) + D([c,a],b) = 0.$$

\vspace{0.3cm}
\noindent
{\bf Cohomology groups of a Lie-Yamaguti algebra:}
K. Yamaguti introduced the following definition of cohomology groups of a Lie-Yamaguti algebra $L$ with coefficients in a representation  $(V; \rho, D, \theta).$
\smallskip 

For $p \ge 1$, let $C^{2p}(L;V)$ denote the space of all $2p$-linear maps  $f: L^{\otimes 2p} \to V$ such that for $i=1,2,\cdots,p$
\[f(x_1,\cdots,x_{2i-1},x_{2i},\cdots,x_{2p}) = 0\]

\medskip
\noindent
whenever  $x_{2i-1}=x_{2i}$.
Also let $C^{2p+1}(L;V)$ denote the space of all $(2p+1)$-linear maps $g: L^{\otimes 2p+1}\to V$ such that for $i=1,2,\cdots,p$
\[g(x_1,\cdots,x_{2i-1},x_{2i},\cdots,x_{2p+1})=0\] 
whenever $x_{2i-1}=x_{2i}.$ 

\medskip
\noindent
Let $C^1(L;V)$ denote the space of all linear maps from $L$ into $V.$ Consider a cochain complex defined as follows:
 
\[\begin{tikzcd}
	{C(L;V)} & {C^{(2, 3)}(L;V)} & {C^{(4, 5)}(L;V)} & \cdots \\
	& {C^{(3, 4)}(L;V) }
	\arrow["\delta", from=1-2, to=1-3]
	\arrow["\delta", from=1-3, to=1-4]
	\arrow["{\delta^*}", from=1-2, to=2-2]
	\arrow["{\delta}", from=1-1, to=1-2]
\end{tikzcd}\]

\medskip\noindent
Here, $C(L;V)$ is the subspace spaned by the diagonal elements $(f,f) \in C^1(L; V) \times C^1(L; V),$ $C^{(3, 4)}(L;V) = C^3(L; V) \times C^4(L; V)$ and for $p\geq 1,$
$$ C^{(2p, 2p+1)}(L; V) = C^{2p}(L;V) \times C^{2p+1}(L;V).$$ 

\medskip\noindent
The coboundary maps are defined as follows:  

\smallskip
For $p \ge 1$; $\delta:C^{(2p, 2p+1)}(L;V) \to C^{(2p+2, 2p+3)}(L;V)$ is of the form $\delta(f,g) = (\delta_If,\delta_{II}g)$
\begin{align*}
	\delta&_If(x_1,x_2,\ldots,x_{2p+2}) \\
	&=
	(-1)^p 
	\Big[\rho(x_{2p+1})g(x_1,\ldots,x_{2p},x_{2p+2}) - \rho(x_{2p+2})g(x_1,\ldots,x_{2p+1}) \\
	&\qquad \qquad\qquad\qquad\qquad\qquad\qquad\qquad\qquad\qquad\quad
	- g(x_1,\ldots,x_{2p},[x_{2p+1},x_{2p+2}])\Big] \\
	&\quad \quad
	+ \sum_{k=1}^p (-1)^{k+1} D(x_{2k-1},x_{2k}) f(x_1,\ldots,\hat{x}_{2k-1}, \hat{x}_{2k},\ldots,x_{2p+2}) \\
	&\quad \quad
	+ \sum_{k=1}^p \sum_{j=2k+1}^{2p+2} (-1)^k f(x_1,\ldots,\hat{x}_{2k-1}, \hat{x}_{2k},\ldots,\{x_{2k-1},x_{2k},x_j\}, \ldots, x_{2p+2}); 
\end{align*}
\begin{align*}	
	\delta&_{II}g(x_1,x_2,\ldots,x_{2p+3}) \\
	&=
	(-1)^p
	\Big[ \theta(x_{2p+2},x_{2p+3})g(x_1,\ldots,x_{2p+1}) - \theta(x_{2p+1},x_{2p+3})g(x_1,\ldots,x_{2p},x_{2p+2})\Big]	\\
	&\quad \quad
	+ \sum_{k=1}^{p+1} (-1)^{k+1} D(x_{2k-1},x_{2k})g(x_1,\ldots,\hat{x}_{2k-1},\hat{x}_{2k},\ldots,x_{2p+3}) \\
	&\quad \quad
	+ \sum_{k=1}^{p+1} \sum_{j=2k+1}^{2p+3} (-1)^k g(x_1,\ldots,\hat{x}_{2k-1},\hat{x}_{2k},\ldots, \{x_{2k-1},x_{2k},x_j\},\ldots,x_{2p+3}).
\end{align*}
\smallskip

\noindent
The map $\delta: C(L;V) \to C^{(2,3)}(L ; V)$ is given by
 $\delta(f,f) = (\delta_If,\delta_{II}f),$ where
 
\begin{equation*}
	\delta_If(a,b) = \rho(a)f(b) - \rho(b)f(a) - f([a,b]), 
\end{equation*}

\begin{equation*}
	\delta_{II}f(a,b,c) = \theta(b,c)f(a) - \theta(a,c)f(b) + D(a,b)f(c) - f(\{a,b,c\}). 
\end{equation*}

\vspace{0.5cm}\noindent
The map  $\delta^{\star}:C^{(2, 3)}(L; V) \to C^{(3, 4)}(L; V) $ is given by $\delta^{\star}(f,g) = (\delta^{\star}_If,\delta^{\star}_{II}g),$ where
\medskip
\begin{align*}
	\delta^{\star}_If(a,b,c) 
	&= -\rho(a)f(b,c) - \rho(b)f(c,a) - \rho(c)f(a,b) + f([a,b],c)  \\
	&\quad \quad
	+ f([b,c],a) + f([c,a],b) + g(a,b,c) + g(b,c,a) + g(c,a,b),
	\\ \\
	\delta^{\star}_{II}g(a,b,c,d) 
	&= \theta(a,d)f(b,c) + \theta(b,d)f(c,a) + \theta(c,d)f(a,b) + g([a,b],c,d)\\ 
	&\quad \quad
	+ g([b,c],a,d) + g([c,a],b,d).
\end{align*}

\medskip
\noindent
Following  \cite{KY-cohomology}, we have for each $f\in C^1(L; V)$
$$\delta_I\delta_I f = \delta^*_I\delta_If= 0 \quad\mbox{and}\quad\delta_{II}\delta_{II}f =  \delta^*_{II}\delta_{II}f= 0.$$
In general, for $(f, g) \in C^{(2p, 2p+1)}(L; V)$ 
$$(\delta \circ \delta)(f, g) = (\delta_I\circ\delta_I(f), \delta_{II}\circ \delta_{II}(g))  =0.$$

\smallskip
The cohomology groups of the Lie-Yamaguti algebra $L$ with coefficients in the representation $(V; \rho, D, \theta)$ is defined as follows.
\medskip
\begin{Def}\label{Definition Cohomology groups}
For $p \ge 2$, let 
$$Z^{(2p, 2p+1)}(L;V) :=span\{(f,g) \in C^{(2p, 2p+1)}(L;V) : \delta(f,g) = 0\}.$$ 

\medskip\noindent
An element of  $Z^{(2p, 2p+1)}(L;V)$ is called a $(2p, 2p+1)$-cocycle. Let 

$$B^{(2p, 2p+1)}(L; V) = \delta(C^{(2p-2, 2p-1)}(L; V)).$$ 

\medskip\noindent
An element of  $B^{(2p, 2p+1)}(L; V)$ is called  a $(2p, 2p+1)$-coboundary. For $p=1$, Let 

$$Z^2(L; V):= span\{f \in C^2(L; V):\delta_If = 0 = \delta^{\star}_If\},$$ 
and
$$Z^3(L;V) := span\{g \in C^3(L; V):\delta_{II}g = 0 = \delta^{\star}_{II}g\}$$ 

\medskip \noindent be the subspaces of $C^2(L; V)$ and $C^3(L; V),$ respectively. Denote by $Z^{(2, 3)}(L; V)$ the subspace $Z^2(L; V) \times Z^3(L;V)$ of $C^{(2, 3)}(L; V).$ 
Let 

$$B^{(2, 3)}(L; V):= B^2(L;V) \times B^3(L; V) = \{\delta (f,f): f \in C^1(L; V)\}.$$

\bigskip \noindent
For $p \ge 1$ the cohomology group $H^{(2p, 2p+1)}(L;V)$ of $L$ with coefficients in $V$ is defined as 

\[H^{(2p,2p+1)}(L; V) := \frac{Z^{(2p, 2p+1)}(L;V)}{B^{(2p, 2p+1)}(L; V)}.\]
\end{Def} 

\vspace{0.5cm}
\section{\Large Deformation of Lie-Yamaguti algebras}\label{$3$}
In this section, we develop one-parameter formal deformation theory of Lie-Yamaguti algebras over a field $\mathbb K$ of characteristic $0.$ We obtain the deformation cohomology that controls deformations in this context by slightly modifying the cohomology of K. Yamaguti as described in the previous section. 
\noindent
Let $K=\mathbb{K}[[t]]$ be the formal power series ring with coefficients in $\mathbb{K}$. Let $(L,[~,~],\{~,~,~\})$ be a Lie-Yamaguti algebra.  Let $L[[t]] := L \otimes_{\mathbb K}K.$ Then, $L[[t]]$ is a module over $K.$ Note that any $n$-linear map 
$$\phi : L \times \cdots \times L \to L$$ extends to an $n$-linear map
$$\phi_t : L[[t]] \times \cdots \times L[[t]] \to L[[t]]$$ over $K.$ Any such extension is said to be {\it ``defined over $K$"}.

\begin{Def}\label{deformation}
A one-parameter formal deformation of the Lie-Yamaguti algebra \\$(L,[~,~],\{~,~,~\})$ is given by a $K$-bilinear map 
$F_t:L[[t]]\times L[[t]]\to L[[t]]$ and a $K$-trilinear map $G_t:L[[t]]\times L[[t]] \times L[[t]] \to L[[t]]$ of the forms

$$F_t=\sum_{i\geq 0} F_it^i \quad \text{and} \quad  G_t=\sum_{i\geq 0}G_it^i,$$
such that
\smallskip
\begin{enumerate}
\item for all $i\geq 0$, $F_i:L\times L\to L$ is a $\mathbb{K}$-bilinear map defined over $K$, and $G_i:L\times L \times L\to L$ is a $\mathbb{K}$-trilinear map defined over $K$;
\smallskip
\item $F_0(~,~)=[~,~]$ is the bilinear bracket of $L$ and $G_0(~, ~, ~)=\{~, ~, ~\}$ is the trilinear bracket of $L$;
\smallskip
\item and $(L[[t]], F_t, G_t)$ is a Lie-Yamaguti algebra over $K.$
\end{enumerate}
{\bf Notation:} We will often use the notation $(F_t, G_t)$ to represent the deformation $(L[[t]], F_t, G_t). $
\end{Def}
\begin{Rem}\label{identity-deformation}
For any Lie-Yamaguti algebra $(L,[~,~],\{~,~,~\}),$ the bilinear bracket $[~,~]$ and the trilinear bracket $\{~,~,~\}$ extend linearly over $K$ giving us a Lie-Yamaguti algebr $L[[t]]$ over $K$ equipped with these extended brackets. This will be referred to as identity deformation. 
\end{Rem}

\smallskip
\noindent
Let $(L[[t]], F_t, G_t)$ be a deformation as above. Then, as $(L[[t]], F_t, G_t)$ is  a Lie-Yamaguti algebra over $K,$ the following equations must hold:

\begin{equation*}\label{deformation-equation}
	\begin{split}
		&F_t(a,a) = 0, \quad G_t(a,a,b) = 0, \\ \\
		&\sum_{\circlearrowleft (a,b,c)} \Big(F_t(F_t(a,b),c) + G_t(a,b,c)\Big) = 0,
		\qquad \sum_{\circlearrowleft (a,b,c)} G_t(F_t(a,b),c,d) = 0, \\ \\
		& G_t(a,b,F_t(c,d)) = F_t(G_t(a,b,c),d) + F_t(c,G_t(a,b,d)), \\ \\
		& G_t(a,b,G_t(c,d,e)) = G_t(g_t(a,b,c),d,e) + G_t(c,G_t(a,b,d),e) + G_t(c,d,G_t(a,b,e)). 
	\end{split}
\end{equation*}

\medskip\noindent
for all $a, b, c, d, e \in L.$
Here, $\sum_{\circlearrowleft (a,b,c)}$ denotes the cyclic sum over $a,b,~\text{and}~c$. Explicitly writing down the above equations, using the facts
\medskip
$$F_t=\sum_{i\geq 0} F_i t^i \quad \text{and} \quad  G_t=\sum_{i \geq 0}G_i t^i,$$ 
\medskip
we obtain the following identities 
\begin{align*}
&~\sum_{i\geq 0} F_{i}(a,a) t^{i} = 0 \\
&~\sum_{{i}\geq 0} G_{i}(a,a,b) t^{i} = 0 \\
&\sum_{\circlearrowleft(a,b,c)} ~\sum_{i,j \ge 0}   \Big(F_i(F_j(a,b),c)+G_{i+j}(a,b,c)\Big)t^{i+j} = 0, \\
&\sum_{\circlearrowleft (a,b,c)}~\sum_{i,j \ge 0}  \Big(G_i(F_j(a,b),c,d)\Big) t^{i+j} = 0, \\
&~\sum_{i,j \geq 0}  
G_i(a,b,F_j(c,d)) t^{i+j} = \sum_{i,j \geq 0} \Big(F_i(G_j(a,b,c),d) + F_i(c,G_j(a,b,d))\Big) t^{i+j}, \\
&~\sum_{i,j \geq 0}  
G_i(a,b,G_j(c,d,e)) t^{i+j} = \sum_{i,j \ge 0} \Big(G_i(G_j(a,b,c),d,e) + G_i(c,G_j(a,b,d),e) \\
& \qquad\qquad\qquad\qquad\qquad\qquad\qquad\qquad\qquad\qquad\qquad\qquad\qquad
+ G_i(c,d,G_j(a,b,e))\Big) t^{i+j}.
\end{align*}

\medskip
Collecting the coefficients of $t^\nu$ for $\nu \geq 0,$ we see that the above equations are equivalent to the following system of equations.

\newpage
\noindent
{\bf Deformation equations for the Lie-Yamaguti algebra $L$:} 

\medskip
\noindent
For $\nu \geq 0$, and for all $a,b,c,d,e \in L$;  
\begin{align}
    &\quad F_{\nu}(a,a) = 0 \label{de1}\\
    &\quad G_{\nu}(a,a,b) = 0, \label{de2}\\
	&\sum_{\substack{i+j=\nu \\ i,j \ge 0}}  \sum_{\circlearrowleft  (a,b,c)} \Big(F_i(F_j(a,b),c)+G_\nu(a,b,c)\Big) = 0, \label{de3}\\
	&\sum_{\substack{i+j=\nu \\ i,j \ge 0}}   \sum_{\circlearrowleft (a,b,c)} \Big(G_i(F_j(a,b),c,d)\Big) = 0, \label{de4}\\
	&\sum_{\substack{i+j=\nu \\ i,j \ge 0}}  
	\Big(G_i(a,b,F_j(c,d)) - F_i(G_j(a,b,c),d) - F_i(c,G_j(a,b,d))\Big)
	= 0, \label{de5}\\
	&\sum_{\substack{i+j=\nu \\ i,j \ge 0}}  
	\Big(G_i(a,b,G_j(c,d,e)) -  G_i(G_j(a,b,c),d,e) - G_i(c,G_j(a,b,d),e) \label{de6} \\
	&\qquad\qquad\qquad\qquad\qquad\qquad\qquad\qquad\qquad\qquad\qquad\quad
	- G_i(c,d,G_j(a,b,e))\Big) = 0. \nonumber
\end{align}

Observe that for $\nu = 0$ equations $(\ref{de1}) - (\ref{de6})$ reduces to the axioms of Lie-Yamaguti algebra $(\ref{LY1})-(\ref{LY6})$. Moreover, the equations $(\ref{de1})-(\ref{de2})$ implies that for each $\nu \geq 0,$ $(F_\nu, G_\nu) \in C^{(2,3)}(L;L).$

\begin{Def}
Let $(L[[t]], F_t, G_t)$ be a deformation as above. The infinitesimal of this formal deformation is the $(2,3)$-cochain $(F_1, G_1).$
More generally, if $F_i = 0 =G_i,~1\leq i \leq n-1,$ and the $(2,3)$-cochain $(F_n, G_n)$ is non-zero, then the $(2,3)$-cochain $(F_n, G_n)$ is called the $n$-infinitesimal of the given deformation.
\end{Def}

Thus, infinitesimal of a deformation is simply the $1$-infinitesimal. The following result relates deformation of a Lie-Yamaguti algebra $L$ with its cohomology $H^{\star}(L; L).$

\begin{Prop} \label{n-infinitesimal-cocycle}
The $n$-infinitesimal of a one-parameter formal deformation $(F_t,G_t)$ of the Lie-Yamaguti algebra $L$ is a $(2,3)$-cocycle.
\end{Prop}

\begin{proof}
In order to show that the $n$-infinitesimal $(F_n,G_n)$ is a $(2,3)$-cocycle, we need to verify $(\delta_I F_n, \delta_{II}G_n) = (0,0)$ and $(\delta_I^\star F_n,\delta_{II}^\star G_n)=(0,0)$.
 
\begin{align*}
\delta_IF_n (a,b,c,d) 
~&=~(-1)^1\Big[\rho(c)G_n(a,b,d) - \rho(d)G_n(a,b,c) - G_n(a,b,[c,d])\Big] + D(a,b)F_n(c,d) \\
~&\qquad 
- F_n(\{a,b,c\},d) - F_n(c,\{a,b,d\}) \\
~&=~ -[c,G_n(a,b,d)] + [d,G_n(a,b,c)] + G_n(a,b,[c,d]) + \{a,b,F_n(c,d)\} \\
~&\qquad
- F_n(\{a,b,c\},d) - F_n(c,\{a,b,d\})\\
~&=~ 0, \qquad(\text{from equation (\ref{de5}) with}~ \nu=n).  
\end{align*}

\begin{align*}
\delta_{II}G_n (a,b,c,d,e)
~&= ~(-1)^1\Big[\theta(d,e)G_n(a,b,c) - \theta(c,e)G_n(a,b,d)\Big] 
+ D(a,b) G_n(c,d,e) \\
~&\qquad
- D(c,d) G_n(a,b,e) - G_n(\{a,b,c\},d,e) 
- G_n(c,\{a,b,d\},e) \\
~& \qquad
- G_n(c,d,\{a,b,e\}) + G_n(a,b,\{c,d,e\})\\
~&=~
- \{G_n(a,b,c),d,e\} - \{c,G_n(a,b,d),e\} 
- \{c,d,G_n(a,b,e)\}  \\
~&\qquad
+ \{a,b,G_n(c,d,e)\} - G_n(\{a,b,c\},d,e) - G_n(c,\{a,b,d\},e) \\
~&\qquad
- G_n(c,d,\{a,b,e\}) + G_n(a,b,\{c,d,e\})\\
~&=~ 0, \qquad (\text{from equation (\ref{de6}) with}~ \nu=n). 
\end{align*}
Similarly by using equation (\ref{de3})  and (\ref{de4}) with $\nu=n$, one can get $\delta_I^\star F_n = 0$ and $\delta_{II}^\star G_n= 0,$ respectively. 
\end{proof}

Next we address the question of extending an infinitesimal deformation to a full-blown deformation. In view of the Proposition \ref{n-infinitesimal-cocycle}, one might be interested to know if every $(2, 3)$-cocycle in $Z^{(2, 3)}(L; L)$ may be realized as the infinitesimal of a formal one-parameter deformation of $L.$ This gives rise to the concept of integrability and to decide integrability of a $(2, 3)$-cocycle, one is led into an obstruction theory, which we will consider in the next section. Before we introduce the definition of integrability of a $(2, 3)$-cocycle, we make the following observations.

\medskip

Let $L$ be a Lie-Yamaguti algebra. Then for any $(2,3)$-cochain $(F_{\nu},G_{\nu}) \in C^{(2,3)}(L;L)$, we have 

\begin{align*}
\delta_I^{\star}F_{\nu}(a,b,c) ~&= \sum_{\substack{i+j=\nu \\ i = 0,\text{or}~ j = 0}}  ~\sum_{\circlearrowleft  (a,b,c)} F_i(F_j(a,b),c)
+\sum_{\circlearrowleft  (a,b,c)} G_{\nu}(a,b,c), \\ \\
\delta_{II}^{\star}G_{\nu} (a,b,c,d) ~&=
\sum_{\substack{i+j=\nu \\i = 0,\text{or}~ j = 0}}   ~\sum_{\circlearrowleft (a,b,c)} G_i(F_j(a,b),c,d), \\ \\ 
\delta_I F_{\nu}(a,b,c,d) ~&=  
\sum_{\substack{i+j=\nu \\ i = 0,\text{or}~ j = 0}}  
\Big(G_i(a,b,F_j(c,d)) - F_i(G_j(a,b,c),d) - F_i(c,G_j(a,b,d))\Big), \\ \\ 
\delta_{II}G_{\nu}(a,b,c,d,e) ~&= \sum_{\substack{i+j=\nu \\ i = 0,\text{or}~ j = 0}}  
\Big(G_i(a,b,G_j(c,d,e)) -  G_i(G_j(a,b,c),d,e) - G_i(c,G_j(a,b,d),e) \\
&\qquad\qquad\qquad\qquad\qquad\qquad\qquad\qquad\qquad\qquad\qquad\qquad
- G_i(c,d,G_j(a,b,e))\Big).
\end{align*}

\bigskip\noindent
Using the above four expressions, equations (\ref{de1})-(\ref{de6}) can be rewritten as

\begin{align}
&\quad F_{\nu}(a,a) ~=~ 0, \label{DE1}\\ 
&\quad G_{\nu}(a,a,b) ~=~ 0, \label{DE2}\\ 
&\sum_{\substack{i+j=\nu \\ i,j > 0}}  ~\sum_{\circlearrowleft  (a,b,c)} F_i(F_j(a,b),c)
~=~ -\delta_I^{\star} F_{\nu}, \label{DE3}\\
&\sum_{\substack{i+j=\nu \\ i,j > 0}}   ~\sum_{\circlearrowleft (a,b,c)} G_i(F_j(a,b),c,d) ~=~ - \delta_{II}^{\star} G_{\nu}, \label{DE4} \\
&\sum_{\substack{i+j=\nu \\ i,j > 0}}  
\Big(G_i(a,b,F_j(c,d)) - F_i(G_j(a,b,c),d) - F_i(c,G_j(a,b,d))\Big) ~=~ -\delta_IF_{\nu}, \label{DE5}\\
&\sum_{\substack{i+j=\nu \\ i,j > 0}}  
\Big(G_i(a,b,G_j(c,d,e)) -  G_i(G_j(a,b,c),d,e) - G_i(c,G_j(a,b,d),e) \label{DE6}\\
&\qquad\qquad\qquad\qquad\qquad\qquad\qquad\qquad\qquad\qquad\quad
- G_i(c,d,G_j(a,b,e))\Big) ~=~ -\delta_{II}G_{\nu}. \nonumber
\end{align}
for all $a,b,c,d,e \in L$.

\begin{Def}
Let $L$ be a Lie-Yamaguti algebra. A $(2,3)$-cocycle is said to be integrable if it is the infinitesimal of a formal one-parameter deformation of $L$.
\end{Def} 

Thus, a $(2,3)$-cocycle $(F_1,G_1)$ is integrable if it can be extended to a sequence of $(2,3)$-cochains, $(F_0,G_0),(F_1,G_1),(F_2,G_2),(F_3,G_3)$, $\ldots$, $(F_{\nu},G_{\nu}),\ldots,$  with $F_0 = [~,~]$ and $G_0 = \{~, ~, ~\}$  so that $L[[t]]$ with 
$$F_t=\sum_{i\geq 0} F_it^i \quad \text{and} \quad  G_t=\sum_{i\geq 0}G_it^i,$$ becomes a formal one-parameter deformations of $L,$ and hence satisfies equations $(\ref{DE3})-(\ref{DE6})$. 

\vspace{0.5cm}
\section{\Large Obstruction cochains and Integrability}\label{$4$}
The obstruction theory plays an important role in deciding whether a given $(2,3)$-cocycle is integrable or not. The aim of this section is to address the question of integrability of a given $(2,3)$-cocycle in $Z^{(2, 3)}(L; L).$ We show in this section that it is necessary to modify the cochain complex of Yamaguti as introduced in Section \ref{$2$} to obtain the right deformation cohomology in the context of deformation of Lie-Yamaguti algebras. We show that there exists a sequence of cohomology classes, vanishing of which is a necessary and sufficient condition for a given $(2,3)$-cocycle to be integrable. We proceed to define obstruction cochains in the present context.

\medskip

Suppose we start with a $n^{th}$ order deformation $(n \geq 1)$, that is, we are given $(2,3)$-cochains $(F_{\nu},G_{\nu})$ such that it satisfy equations $(\ref{DE3})-(\ref{DE6})$ for $1\leq \nu \leq n$. In order to extend the given $n^{th}$ order deformation to an $(n+1)^{th}$ order deformation we must have a  cochain $(F_{n+1}, G_{n+1}) \in C^{(2, 3)}(L; L)$ such that the equations $(\ref{DE3})-(\ref{DE6})$ hold for $1\leq \nu \leq n+1.$ In other words, for $a, b, c, d, e \in L,$ the following equations should hold.

\begin{align*}
&\sum_{\substack{i+j=n+1 \\ i,j > 0}}  ~\sum_{\circlearrowleft  (a,b,c)} F_i(F_j(a,b),c)
 ~=~ -\delta_I^{\star} F_{n+1}(a,b,c),\\
&\sum_{\substack{i+j=n+1 \\ i,j > 0}}   ~\sum_{\circlearrowleft (a,b,c)} G_i(F_j(a,b),c,d) ~=~ - \delta_{II}^{\star} G_{n+1}(a,b,c,d),\\
&\sum_{\substack{i+j=n+1 \\ i,j > 0}}  
\Big(G_i(a,b,F_j(c,d)) - F_i(G_j(a,b,c),d) - F_i(c,G_j(a,b,d))\Big) ~=~ -\delta_IF_{n+1}(a,b,c,d), \\
&\sum_{\substack{i+j=n+1 \\ i,j > 0}}  
\Big(G_i(a,b,G_j(c,d,e)) -  G_i(G_j(a,b,c),d,e) - G_i(c,G_j(a,b,d),e)\\
&\qquad\qquad\qquad\qquad\qquad\qquad\qquad\qquad\qquad\qquad
- G_i(c,d,G_j(a,b,e))\Big) ~=~ -\delta_{II}G_{n+1}(a,b,c,d,e). 
\end{align*}

\bigskip\noindent
Let $a, b, c, d, e \in L.$ Define  $P, Q, R,~ \text{and}~ S$ as follows

\begin{align}
P (a,b,c) & = \sum_{\substack{i+j=n+1 \\ i,j \geq 1}} \sum_{\circlearrowleft (a,b,c)} F_i(F_{j}(a,b),c),\label{O1}\\
Q (a,b,c,d) & = \sum_{\substack{i+j=n+1 \\ i,j \geq 1}} \sum_{\circlearrowleft (a,b,c)} G_i(F_{j}(a,b),c,d),\label{O2}\\
R (a,b,c,d) & = \sum_{\substack{i+j=n+1 \\ i,j \geq 1}} \Big(G_i(a,b,F_{j}(c,d)) - F_i(G_{j}(a,b,c),d) - F_i(c,G_{j}(a,b,d))\Big), \label{O3}\\
S (a,b,c,d,e) & = \sum_{\substack{i+j=n+1 \\ i,j \ge 1}} \label{O4}
\Big(G_i(a,b,G_{j}(c,d,e)) - G_i(G_{j}(a,b,c),d,e)\\
&\qquad\qquad\qquad\qquad\qquad\quad
- G_i(c,G_{j}(a,b,d),e) - G_i(c,d,G_{j}(a,b,e)) \Big). 	\nonumber
\end{align}

\bigskip\noindent
Note that $(P,Q) \in C^{(3,4)}(L;L)$ and $(R,S) \in C^{(4,5)}(L;L)$.  We next observe that $(R, S)$ is a $(4,5)$-cocycle.

\begin{Lem} \label{obstruction_cocycle}
Let $L$ be a Lie-Yamaguti algebra and $(F_1,G_1),(F_2,G_2),\ldots,(F_n,G_n)$ define an $n^{th}$ order deformation of $L$. Then $(R,S)$ defined in (\ref{O3}) and (\ref{O4}) is a $(4,5)$-cocycle. 
\end{Lem}

\noindent
In order to prove this, we first define a circle product as follows.

\bigskip\noindent
The circle product

\[\circ = (\circ_1,\circ_2) : C^{(2k,2k+1)}(L; L) \times C^{(2l,2l+1)}(L; L) \longrightarrow C^{(2k+2l,2k+2l+1)}(L; L),\]

\medskip\noindent
denoted by 
\bigskip
\[(f_1,g_1) \circ (f_2,g_2) = \Big((f_1,g_1)\circ_1 (f_2,g_2) ~,~ (f_1,g_1)\circ_2(f_2,g_2)\Big)\]

\bigskip\noindent
is defined for the following values of $k$ and $l$: $k=1,l=1;~ k=1,l=2; ~\text{and}~k=2, l=1.$ Let $a_i \in L,~1\leq i \leq 7.$

\bigskip\noindent
Let $k=1$ and $l=1.$ Then, for $(f_1,g_1),~(f_2,g_2)\in C^{(2,3)}(L; L)$ define
\bigskip
\begin{align*}
& (f_1,g_1) \circ_1 (f_2,g_2) (a_1,a_2,a_3,a_4) \\
&\qquad =  g_1(a_1,a_2,f_2(a_3,a_4)) - f_1(g_2(a_1,a_2,a_3),a_4) - f_1(a_3,g_2(a_1,a_2,a_4)),\\ \\
& (f_1,g_1) \circ_2 (f_2,g_2) (a_1,a_2,a_3,a_4,a_5) \\
&\qquad = g_1(a_1,a_2,g_2(a_3,a_4,a_5)) - g_1(g_2(a_1,a_2,a_3),a_4,a_5) - g_1(a_3,g_2(a_1,a_2,a_4),a_5)\\
& \qquad\qquad  - g_1(a_3,a_4,g_2(a_1,a_2,a_5)). 	
\end{align*}

\bigskip\noindent
Let $k=1$ and $l=2.$ For $(f_1,g_1) \in C^{(2,3)}(L;L)$ and $(f_2,g_2) \in C^{(4,5)}(L;L)$  define
\bigskip
\begin{align*}
& (f_1,g_1) \circ_1 (f_2,g_2) (a_1,a_2,a_3,a_4,a_5,a_6) \\
&\qquad = -f_1(g_2(a_1,a_2,a_3,a_4,a_5),a_6) - f_1(a_5,g_2(a_1,a_2,a_3,a_4,a_6)) \\ 
& \qquad\qquad+ g_1(a_3,a_4,f_2(a_1,a_2,a_5,a_6)) - g_1(a_1,a_2,f_2(a_3,a_4,a_5,a_6)), \\ \\
& (f_1,g_1) \circ_2 (f_2,g_2) (a_1, a_2,a_3,a_4,a_5,a_6,a_7) \\
& \qquad= -g_1(g_2(a_1,a_2,a_3,a_4,a_5),a_6,a_7)  - g_1(a_5,g_2(a_1,a_2,a_3,a_4,a_6),a_7) \\
& \qquad\qquad- g_1(a_5,a_6,g_2(a_1,a_2,a_3,a_4,a_7)) + g_1(a_3,a_4,g_2(a_1,a_2,a_5,a_6,a_7)) \\
& \qquad\qquad
- g_1(a_1,a_2,g_2(a_3,a_4,a_5,a_6,a_7)).
\end{align*}

\bigskip\noindent
Let $k=2$ and $l=1.$ Then, for $(f_1,g_1) \in C^{(4,5)}(L;L)$ and $(f_2,g_2) \in C^{(2,3)}(L;L)$ define
\bigskip
\begin{align*}
& (f_1,g_1) \circ_1 (f_2,g_2) (a_1,a_2,a_3, a_4,a_5,a_6) \\
&\qquad = f_1(g_2(a_1,a_2,a_3),a_4,a_5,a_6)  + f_1(a_3,g_2(a_1,a_2,a_4),a_5,a_6) \\
& \qquad\qquad+ f_1(a_3,a_4,g_2(a_1,a_2,a_5),a_6) +f_1(a_3,a_4,a_5,g_2(a_1,a_2,a_6)) \\
& \qquad\qquad- f_1(a_1,a_2,g_2(a_3,a_4,a_5),a_6) - f_1(a_1,a_2,a_5,g_2(a_3,a_4,a_6)) \\
& \qquad\qquad+g_1(a_1,a_2,a_3,a_4,f_2(a_5,a_6)),	
\end{align*}

\begin{align*}
& (f_1,g_1) \circ_2 (f_2,g_2) (a_1,a_2,a_3,a_4,a_5,a_6,a_7) \\
& \qquad = g_1(g_2(a_1,a_2,a_3),a_4,a_5,a_6,a_7) + g_1(a_3,g_2(a_1,a_2,a_4),a_5,a_6,a_7) \\
& \qquad\qquad + g_1(a_3,a_4,g_2(a_1,a_2,a_5),a_6,a_7) + g_1(a_3,a_4,a_5,g_2(a_1,a_2,a_6),a_7) \\
& \qquad\qquad + g_1(a_3,a_4,a_5,a_6,g_2(a_1,a_2,a_7)) - g_1(a_1,a_2,g_2(a_3,a_4,a_5),a_6,a_7) \\
& \qquad\qquad - g_1(a_1,a_2,a_5,g_2(a_3,a_4,a_6),a_7) - g_1(a_1,a_2,a_5,a_6,g_2(a_3,a_4,a_7)) \\
& \qquad\qquad + g_1(a_1,a_2,a_3,a_4,g_2(a_5,a_6,a_7)).
\end{align*}

\noindent
Now observe that the cochain $(R,S) \in C^{(4,5)}(L;L)$ may be expressed in terms of the circle product defined above as follows. 
\begin{align*}
	R (a_1,&a_2,a_3,a_4) \\
	&= \sum_{\substack{i+j=n+1 \\ i,j \ge 1}} 
	\Big(G_i(a_1,a_2,F_{j}(a_3,a_4)) - F_i(G_{j}(a_1,a_2,a_3),a_4) 
	- F_i(a_3,G_{j}(a_1,a_2,a_4))\Big) \\
	&= \sum_{\substack{i+j=n+1 \\ i,j \ge 1}}
	(F_i,G_i) \circ_1 (F_j,G_j) (a_1,a_2,a_3,a_4).	\\ \\
    S (a_1,&a_2,a_3,a_4,a_5) \\
	&=
	\sum_{\substack{i+j=n+1 \\ i,j \ge 1}}
	\Big(G_i(a_1,a_2,G_j(a_3,a_4,a_5)) 
	- G_i(G_j(a_1,a_2,a_3),a_4,a_5) \\
	&\qquad\qquad\qquad\qquad\qquad
	- G_i(a_3,G_j(a_1,a_2,a_4),a_5) 
	- G_i(a_3,a_4,G_j(a_1,a_2,a_5))\Big) \\ \\
	&= 
	\sum_{\substack{i+j=n+1 \\ i,j \ge 1}}
	(F_i,G_i) \circ_2 (F_j,G_j) (a_1,a_2,a_3,a_4,a_5).
\end{align*}

\bigskip\noindent
Thus, the cochain $(R, S) \in C^{(4,5)}(L;L)$ is given by

\begin{equation*}
	R ~= \sum_{\substack{i+j=n+1 \\ i,j \geq 1}}(F_i,G_i) \circ_1 (F_j,G_j), \quad
	S ~= \sum_{\substack{i+j=n+1 \\ i,j \geq 1}}(F_i,G_i) \circ_2 (F_j,G_j).
\end{equation*}

\bigskip\noindent
Now we are in a position to prove the Lemma \ref{obstruction_cocycle}. 
\begin{proof}[Proof of Lemma \ref{obstruction_cocycle}]
\smallskip\noindent
A direct computation gives us the following identity: For any two element $(F_i,G_i),(F_j,G_j) \in C^{(2,3)}(L;L)$ we have   
\begin{multline*}
	\delta \Big((F_i,G_i)\circ_1(F_j,G_j)~,~(F_i,G_i)\circ_2(F_j,G_j)\Big)
	\\
	= \Big(\delta_I\big((F_i,G_i)\circ_1(F_j,G_j)\big)~,~ \delta_{II} \big((F_i,G_i)\circ_2(F_j,G_j)\big) \Big)
\end{multline*} 
where 
\begin{align*}
	\delta_I \big((F_i,G_i)&\circ_1(F_j,G_j)\big) (a_1,a_2,a_3,a_4,a_5,a_6)\\ 
	&= 
	\delta(F_i,G_i) \circ_1 (F_j,G_j) (a_1,\ldots,a_6)  
	+ (F_i,G_i) \circ_1 \delta (F_j,G_j) (a_1,\ldots,a_6) \\
	&\qquad
	- [G_i(a_3,a_4,a_6),G_j(a_1,a_2,a_5)] 
	+ [G_j(a_3,a_4,a_6),G_i(a_1,a_2,a_5)] \\
	&\qquad
	+ [G_i(a_3,a_4,a_5),G_j(a_1,a_2,a_6)]
	- [G_j(a_3,a_4,a_5),G_i(a_1,a_2,a_6)] \\
	&\qquad
	+ \{G_i(a_1,a_2,a_3),a_4,F_j(a_5,a_6)\} 
	- \{G_j(a_1,a_2,a_3),a_4,F_i(a_5,a_6)\} \\
	&\qquad
	- \{G_i(a_1,a_2,a_4),a_3,F_j(a_5,a_6)\}
	+ \{G_j(a_1,a_2,a_4),a_3,F_i(a_5,a_6)\} \\ \\
	\delta_{II}\big((F_i,G_i)& \circ_2 (F_j,G_j)\big) (a_1,a_2,a_3,a_4,a_5,a_6,a_7) \\
	&=
	\delta(F_i,G_i) \circ_2 (F_j,G_j) G_j(a_1,a_2,\ldots,a_7) + (F_i,G_i) \circ_2 \delta (F_j,G_j) (a_1,a_2\ldots,a_7) \\
	&\qquad
	-\{G_i(a_1,a_2,a_3),a_4,G_j(a_5,a_6,a_7)\} 
	+ \{G_j(a_1,a_2,a_3),a_4,G_i(a_5,a_6,a_7)\} \\
	&\qquad
	- \{a_3,G_i(a_1,a_2,a_4),G_j(a_5,a_6,a_7)\} 
	+ \{a_3,G_j(a_1,a_2,a_4),G_i(a_5,a_6,a_7)\} \\
	&\qquad
	+ \{G_i(a_1,a_2,a_5),a_6,G_j(a_3,a_4,a_7)\} 
	- \{G_j(a_1,a_2,a_5),a_6,G_i(a_3,a_4,a_7)\} \\
	&\qquad
	+ \{G_i(a_1,a_2,a_5),G_j(a_3,a_4,a_6),a_7\} 
	- \{G_j(a_1,a_2,a_5),G_i(a_3,a_4,a_6),a_7\} \\
	&\qquad
	- \{G_i(a_1,a_2,a_6),a_5,G_j(a_3,a_4,a_7)\}
	+ \{G_j(a_1,a_2,a_6),a_5,G_i(a_3,a_4,a_7)\} \\
	&\qquad
	- \{G_i(a_1,a_2,a_6),G_j(a_3,a_4,a_5),a_7\}
	+ \{G_j(a_1,a_2,a_6),G_i(a_3,a_4,a_5),a_7\} \\
	&\qquad
	- \{G_i(a_3,a_4,a_5),a_6,G_j(a_1,a_2,a_7)\} 
	+ \{G_j(a_3,a_4,a_5),a_6,G_i(a_1,a_2,a_7)\} \\
	&\qquad
	- \{a_5,G_i(a_3,a_4,a_6),G_j(a_1,a_2,a_7)\} 
	+ \{a_5,G_j(a_3,a_4,a_6),G_i(a_1,a_2,a_7)\}. 
\end{align*}

\medskip\noindent
With the above identities in hand, we compute $\delta(R,S) = (\delta_IR, \delta_{II}S).$ Note that
\medskip
\begin{align*}
	\delta_I R 
	~~&=  \sum_{\substack{i+j=n+1 \\ i,j \ge 1}}
	\delta_I\big((F_i,G_i) \circ_1 (F_j,G_j)\big) \\
	&= \sum_{\substack{i+j=n+1 \\ i,j \ge 1}}
	\delta(F_i,G_i) \circ_1 (F_j,G_j) 
	+ (F_i,G_i) \circ_1 \delta (F_j,G_j) \\
	&= \sum_{\substack{i+j+k=n+1 \\ i,j \ge 1}} 	
	\Big(\big((F_i,G_i) \circ_1 (F_j,G_j)\big) \circ_1 (F_k,G_k) 
	+ (F_i,G_i) \circ_1\big((F_j,G_j) \circ_1 (F_k,G_k)\big)\Big)	
\end{align*}
\begin{align*}
	\delta_{II}S 
	~~&= 
	\sum_{\substack{i+j=n+1 \\ i,j \ge 1}}
	\delta_{II}\big((F_i,G_i) \circ_2 (F_j,G_j)\big) \\
	&= 
	\sum_{\substack{i+j=n+1 \\ i,j \ge 1}} 
	\delta(F_i,G_i) \circ_2 (F_j,G_j) 
	+ (F_i,G_i) \circ_2 \delta (F_j,G_j) \\
	&= \sum_{\substack{i+j+k=n+1 \\ i,j \ge 1}} 	
	\Big(\big((F_i,G_i) \circ_2 (F_j,G_j)\big) \circ_2 (F_k,G_k) 
	+ (F_i,G_i) \circ_2 \big((F_j,G_j) \circ_2 (F_k,G_k)\big)\Big).
\end{align*}

\bigskip\noindent
Next, we compute the summands 
\medskip 
\[\Big(\big((F_i,G_i) \circ_1 (F_j,G_j)\big) \circ_1 (F_k,G_k) + (F_i,G_i) \circ_1\big((F_j,G_j) \circ_1 (F_k,G_k)\big)\Big)\]
and 
\[\Big(\big((F_i,G_i) \circ_2 (F_j,G_j)\big) \circ_2 (F_k,G_k) + (F_i,G_i) \circ_2 \big((F_j,G_j) \circ_2 (F_k,G_k)\big)\Big)\]

\bigskip\noindent
appeared on the right hand side of $\delta_I R$ and $\delta_{II}S$ above, respectively. An explicit computation shows  
\medskip
\begin{align*}
	\big(\big((F_i,G_i) &\circ_1 (F_j,G_j)\big) \circ_1 (F_k,G_k) 
	+ (F_i,G_i) \circ_1\big((F_j,G_j) \circ_1 (F_k,G_k)\big)\big)	
	(a_1,a_2,a_3,a_4,a_5,a_6) \\ 
	&=G_i(G_k(a_1,a_2,a_3),a_4,F_j(a_5,a_6)) 
	- G_i(G_j(a_1,a_2,a_3),a_4,F_k(a_5,a_6)) \\
	&\qquad
	+ G_i(a_3,G_k(a_1,a_2,a_4),F_j(a_5,a_6))
	- G_i(a_3,G_j(a_1,a_2,a_4),F_k(a_5,a_6)) \\
	&\qquad
	- F_i(G_k(a_1,a_2,a_5),G_j(a_3,a_4,a_6)) 
	+ F_i(G_j(a_1,a_2,a_5),G_k(a_3,a_4,a_6)) \\
	&\qquad
	+ F_i(G_k(a_3,a_4,a_5),G_j(a_1,a_2,a_6))
	- F_i(G_j(a_3,a_4,a_5),G_k(a_1,a_2,a_6))
\end{align*}
and
\medskip
\begin{align*}
	\Big(\big((F_i,G_i) &\circ_2 (F_j,G_j)\big) \circ_2 (F_k,G_k) + (F_i,G_i) \circ_2 \big((F_j,G_j) \circ_2 (F_k,G_k)\big)\Big) (a_1,a_2,a_3,a_4,a_5,a_6,a_7) \\ 
	&=
	G_i(G_j(a_1,a_2,a_3),a_4,G_k(a_5,a_6,a_7)) 
	- G_i(G_k(a_1,a_2,a_3),a_4,G_j(a_5,a_6,a_7)) \\
	&\qquad
	+ G_i(a_3,G_j(a_1,a_2,a_4),G_k(a_5,a_6,a_7)) 
	- G_i(a_3,G_k(a_1,a_2,a_4),G_j(a_5,a_6,a_7)) \\
	&\qquad
	- G_i(G_j(a_1,a_2,a_5),a_6,G_k(a_3,a_4,a_7)) 
	+ G_i(G_k(a_1,a_2,a_5),a_6,G_j(a_3,a_4,a_7)) \\
	&\qquad
	- G_i(G_j(a_1,a_2,a_5),G_k(a_3,a_4,a_6),a_7) 
	+ G_i(G_k(a_1,a_2,a_5),G_j(a_3,a_4,a_6),a_7) \\
	&\qquad
	+ G_i(G_j(a_1,a_2,a_6),a_5,G_k(a_3,a_4,a_7))
	- G_i(G_k(a_1,a_2,a_6),a_5,G_j(a_3,a_4,a_7)) \\
	&\qquad
	+ G_i(G_j(a_1,a_2,a_6),G_k(a_3,a_4,a_5),a_7) 
	- G_i(G_k(a_1,a_2,a_6),G_j(a_3,a_4,a_5),a_7) \\
	&\qquad
	+ G_i(G_j(a_3,a_4,a_5),a_6,G_k(a_1,a_2,a_7)) 
	- G_i(G_k(a_3,a_4,a_5),a_6,G_j(a_1,a_2,a_7)) \\
	&\qquad
	+ G_i(a_5,G_j(a_3,a_4,a_6),G_k(a_1,a_2,a_7))
	- G_i(a_5,G_k(a_3,a_4,a_6),G_j(a_1,a_2,a_7)).
\end{align*}

\bigskip\noindent
Putting this identities in the expression of $\delta_I R$ and $\delta_{II} S$ and simplifying we obtain \\$\delta (R,S) = 0$. 
Hence,  $\big(R,S\big) \in Z^{(4,5)}(L;L),$ that is, $(R,S)$ is a $(4,5)$-cocycle.
\end{proof}

\medskip\noindent
Thus, we see that 
\medskip
$$\big((P, Q), (R, S )\big) \in C^{(3, 4)}(L; L)\times C^{(4,5)}(L;L)$$ 

\medskip\noindent
is a cochain with $(R, S) \in  C^{(4,5)}(L;L)$ being a cocycle. This cochain appears as an obstruction to extend a given deformation of order $n$ to a deformation of the next higher order. This motivates us to define  the following cochain complex defining the deformation cohomology.

\begin{Def}\label{deformation-cohomology}
Consider the cochain complex
\bigskip
\[\begin{tikzcd}
C(L;L) & {C^{(2,3)}(L;L)} & {C^{(3,4)}(L;L) \times C^{(4,5)}(L;L)} & {C^{(6,7)}(L;L)} & \cdots
\arrow["\delta", from=1-1, to=1-2]
\arrow["{(\delta^*,~ \delta)}", from=1-2, to=1-3]
\arrow["{(0,\delta)}", from=1-3, to=1-4]
\arrow["\delta", from=1-4, to=1-5]
\end{tikzcd}\] 

\bigskip\noindent
where the cochain groups and the coboundary maps are the same as defined in Section \ref{$2$}. This is called the deformation complex and the cohomology groups defined by it are the deformation cohomology groups.
\end{Def}
\noindent 
Note that $\mbox{Ker}~(\delta^\star, \delta) = Z^{(2, 3)}(L; L)$ and $\mbox{Img}~(\delta) = B^{(2,3)}(L; L).$
Thus, the deformation cohomology groups important in the present context are

$$H^{(2, 3)}(L; L) := \frac{Z^{(2, 3)}(L; L)}{B^{(2, 3)}(L; L)}$$ and  
$$H^{((3, 4), (4, 5))}(L; L) := \frac{\mbox{Ker}~(0, \delta)}{\mbox{Img}~(\delta^*,~ \delta)}.$$ 

\bigskip\noindent
Clearly, by construction of this cochain complex and Lemma \ref{obstruction_cocycle} 
\medskip
$$\big((P, Q), (R, S)\big) \in \mbox{Ker}~(0, \delta)$$ 

\smallskip\noindent
and hence, represent a cohomology class in $H^{((3, 4), (4, 5))}(L; L).$ We now state below one of the main theorem of the paper.

\begin{Thm} \label{obstruction_vanish}
Let $(L, [~, ~], \{~, ~, ~\})$ be a Lie-Yamaguti algebra. Let $(F_t, G_t)$ be a deformation of $L$ of order $n$ given by a sequence
$(F_1, G_1), \ldots, (F_n, G_n)$ of $(2,3)$-cochains satisfying  equations $(\ref{DE3})-(\ref{DE6})$ for $1\leq \nu \leq n.$ Then, the cohomology class of $\big((P,Q), (R,S)\big)$ vanishes if and only if the given $n^{th}$ order deformation extends to a $(n+1)^{th}$ order deformation, that is, to a sequence $(F_1,G_1),(F_2,G_2),\ldots,(F_n,G_n), (F_{n+1},G_{n+1})$ satisfying $(\ref{DE3})-(\ref{DE6})$ for all $\nu =1,2,\ldots, n+1$.
\end{Thm}

\begin{proof}
Suppose we have a deformation of order $n \geq 1$ given by a sequence $(F_1, G_1), \ldots, (F_n, G_n)$ of $(2,3)$-cochains satisfying  equations (\ref{DE3})-(\ref{DE6}) for $1\leq \nu \leq n.$
\medskip\noindent
Rewriting equation (\ref{DE3})-(\ref{DE6}) for $\nu =n+1$ we get,
\begin{align*}
	P(a,b,c) ~=~ - \delta_I^* F_{n+1}(a,b,c), &\quad Q(a,b,c,d) ~=~ - \delta_{II}^*G_{n+1}(a,b,c,d),\\
	R(a,b,c,d) ~=~ - \delta_IF_{n+1}(a,b,c,d), &\quad 
	S(a,b,c,d,e) ~=~ - \delta_{II}G_{n+1} (a,b,c,d,e),
\end{align*}
$a, b, c, d, e \in L.$
\medskip\noindent
Now if the cohomology class of $\big((P,Q), (R,S)\big)$ vanishes, then there exists $(F_{n+1},G_{n+1}) \in C^{(2,3)}(L;L)$ such that 
$$\delta^*(F_{n+1},G_{n+1})  ~=~ -(P,Q), \quad \delta(F_{n+1},G_{n+1}) ~=~ -(R,S).$$
As a result, we have a $(2,3)$-cochain, $(F_{n+1},G_{n+1}) \in C^{(2,3)}(L;L)$ satisfying (\ref{DE3}), (\ref{DE4}), (\ref{DE5}), and (\ref{DE6}) for $\nu =n+1$. Hence, one can extend an $n^{th}$ order deformation to an $(n+1)^{th}$ order deformation.

\bigskip\noindent
Conversely, let $(F_1,G_1),\ldots,(F_n,G_n)$ be a given $n^{th}$ order deformation of $L$ that can be extended to a $(n+1)^{th}$ order deformation. Let $(F_{n+1}, G_{n+1})$ be a $(2, 3)$-cochain such that the $(2, 3)$ cochains $(F_\nu, G_\nu)$ satisfy (\ref{DE3})-(\ref{DE6}) for $1 \leq \nu \leq n+1.$ This would force the corresponding obstruction cochain $\big((P,Q), (R,S)\big)$ to be a coboundary. Hence the cohomology class of $\big((P,Q), (R,S)\big)$ vanishes. 
\end{proof}

\medskip

It follows that with $F_0 = [~,~],$ $G_0= \{~,~,~\}$ being the original Lie-Yamaguti algebra structure, $\sum F_it^i$ is a deformation modulo $t^{n+1}$ if and only if the deformation equations hold for $1 \le \nu \le n$ and this can be extended to a deformation modulo $t^{n+2}$ if and only if the cohomology class of $\big((P,Q), (R,S)\big)$ vanishes. 

\begin{Def}
The cohomology class of $\big((P,Q), (R,S)\big)$ is called the {\bf obstruction class} to extend the sequence $(F_1,G_1),(F_2,G_2),\ldots,(F_n,G_n)$ satisfying equations (\ref{DE3})-(\ref{DE6}), for $1 \le \nu \le n$ to a sequence $(F_1,G_1),(F_2,G_2),\ldots,(F_n,G_n), (F_{n+1},G_{n+1})$ satisfying equations (\ref{DE3})-(\ref{DE6}), for $1 \le \nu \le n+1$.
\end{Def} 

As a consequence of Theorem \ref{obstruction_vanish}, we have the following sufficient condition for integrability of every $(2,3)$-cocycle.

\begin{Cor}\label{Int_thm}
Let $(L,[~,~],\{~,~,~\})$ be a Lie-Yamaguti algebra such that
\[H^{((3,4), (4,5))}(L; L) = 0.\]  	Then all the obstruction classes vanish and hence any $(2,3)$-cocycle $(F_1,G_1) \in Z^{(2,3)}(L; L)$ is integrable.  
\end{Cor}  

\section{\Large Equivalent deformations and Rigidity}\label{$5$}
In this section, we introduce the notion of equivalence of deformations and rigid object.

Let $(L,[~,~],\{~,~,~\})$ be a Lie-Yamaguti algebra over $\mathbb K.$ Let $L[[t]] := L \otimes_{\mathbb K}K.$ Note that any $\mathbb K$-linear map $\phi : L   \to L$ extends to a $K$-linear map $\phi_t : L[[t]] \to L[[t]]$ over $K.$ Any such extension is said to be {\it defined over $K$}. Let  $(L[[t]], F_t, G_t)$  and  $(L[[t]], F^\prime_t, G^\prime_t)$ be two formal one-parameter deformations of $L$ given by
\[F_t(a,b) := F_0(a,b) + F_1(a,b)t + F_2(a,b)t^2 + \cdots \qquad a,b \in L\]
\[G_t(a,b,c) := G_0(a,b,c) + G_1(a,b,c)t + G_2(a,b,c)t^2+\cdots \qquad a,b,c \in L,\] where $F_0= [~,~],$ and  $G_0 = \{~,~,~\}$ are the original brackets on $L.$  Note that any $K$-linear homomorphism $L[[t]] \to L[[t]]$ is of the form $\sum_i\psi_it^i,$ where $\psi_i \in \mbox{Hom}_{\mathbb K}(L, L).$

\begin{Def}
A (formal) isomorphism $\Psi_t: (L[[t]], F_t, G_t) \to (L[[t]], F^\prime_t, G^\prime_t)$ is a $K$-Linear automorphism 
$$ \Psi_t(a) = a + \sum_{i\geq 1}\psi_i(a),$$ where each $\psi_i: L[[t]] \to L[[t]]$ is a linear map defined over $K$ such that 
\[\Phi_t (F_t(a,b)) = F^\prime_t\left(\Phi_t(a),\Phi_t(b)\right) \quad \text{and} \quad \Phi_t ( G_t(a,b,c)) = G^\prime_t\left(\Phi_t(a),\Phi_t(b),\Phi_t(c)\right) \]
for all $a, b, c \in L.$ 
\end{Def}

\begin{Def}
Let $(L,[~,~],\{~,~,~\})$ be a Lie-Yamaguti algebra over $\mathbb K.$ Two one-parameter formal deformations  $(L[[t]], F_t, G_t)$  and  $(L[[t]], F^\prime_t, G^\prime_t)$  of a Lie-Yamaguti algebra $L$ are said to be equivalent, denoted by $(F_t,G_t) \sim (F^\prime_t, G^\prime_t)$, if there exists a formal isomorphism of $K$-modules 
\[\Phi_t: (L[[t]],F_t,G_t) \to (L[[t]],F^\prime_t, G^\prime_t)\] as defined above.
\end{Def}

\begin{Def}
Any deformation $(L[[t]], F_t,G_t)$ equivalent to the identity deformation \ref{identity-deformation} is called a {\bf trivial deformation}.
\end{Def}

The relation between infinitesimals of two equivalent deformations is given by the following proposition. 

\begin{Prop} Let $(L[[t]],F_t,G_t)$ and  $(L[[t]],F^\prime_t, G^\prime_t)$ be two equivalent one-parameter formal deformations of a Lie-Yamaguti algebra $L$. Then the corresponding infinitesimals $(F_1,G_1)$ and $(F^\prime_1, G^\prime_1)$ determine the same cohomology class in $H^{(2,3)}(L; L)$. 
\end{Prop}

\begin{proof}
Given $(F_t, G_t) \sim (F^\prime_t, G^\prime_t).$ Let $\Phi_t(a)= \sum_{i \geq 0} \phi_i(a) t^i$  be a formal isomorphism
\medskip
\[\Phi_t: (L[[t]],F_t,G_t) \to (L[[t]],F^\prime_t, G^\prime_t).\] Then, for any $a,b,c \in L$ we have  

\[\sum_{i \geq 0} \phi_i\Big(\sum_{j \geq 0} F_j(a,b)t^j\Big)t^i 
~=~
\sum_{i \geq 0} F^\prime_i\Big(\sum_{k \geq 0} \phi_k(a)t^k,\sum_{l \geq 0}\phi_l(b)t^l\Big)t^i\]
and 
\[\sum_{i \geq 0} \phi_i\Big(\sum_{j \geq 0} G_j(a,b,c)t^j\Big)t^i
~=~
\sum_{i \geq 0} G^\prime_i\Big(\sum_{k \geq 0}\phi_k(a)t^k,\sum_{l \geq 0}\phi_l(b)t^l,\sum_{m \geq 0}\phi_m(c)t^m\Big)t^i.\]

\medskip\noindent
Comparing the coefficients of $t^1$ from the above equations we have
$$F_1(a,b) + \phi_1([a,b]) = F^\prime_1(a,b) + [\phi_1(a),b] + [a,\phi_1(b)].$$
Equivalently,
\begin{align*} 
F_1(a,b) - F^\prime_1(a,b) ~&=~ [a,\phi_1(b)] + [\phi_1(a),b] - \phi_1([a,b]) \\
~&=~ \delta_I \phi_1 (a,b)
\end{align*}
and 
$$G_1(a,b,c) + \phi_1(\{a,b,c\}) = G^\prime_1(a,b,c) + \{\phi_1(a),b,c\} + \{a,\phi_1(b),c\} + \{a,b,\phi_1(c)\}.$$
Equivalently,
\begin{align*}
G_1(a,b,c) - G^\prime_1(a,b,c)
~&=~
\{\phi_1(a),b,c\} + \{a,\phi_1(b),c\} + \{a,b,\phi_1(c)\} - \phi_1(\{a,b,c\}) \\ 
~&=~ \delta_{II}\phi_1(a,b,c).
\end{align*}

\noindent
In other words,
$$(F_1-F^\prime_1, G_1-G^\prime_1)= \delta \phi \in B^{(2,3)}(L; L).$$ Thus $(F_1,G_1)$ and $(F^\prime_1, G^\prime_1)$ are cohomologous and hence, determine the same cohomology class. 
\end{proof}

On the other hand, if two $(2, 3)$-cocycles are cohomologous and one is the infinitesimal of a deformation, then so is the other of an equivalent deformation.

\begin{Prop}
Let $(L[[t]], F_t, G_t)$ be a deformation of a Lie-Yamaguti algebra $(L, [~,~], \{~,~,~\}),$ with the infinitesimal $(F_1, G_1)$ where
$$F_t=\sum_{i\geq 0} F_it^i \quad \text{and} \quad  G_t=\sum_{i\geq 0}G_it^i.$$ Let $(F^\prime_1, G^\prime_1)$ be a $(2, 3)$-cocycle cohomologous to $(F_1, G_1).$ Then there exists a deformation of $L$ equivalent to  $(L[[t]], F_t, G_t)$ with infinitesimal $(F^\prime_1, G^\prime_1).$
\end{Prop}

\begin{proof}
By the hypothesis, there exists some $\lambda \in C^1(L;L)$ such that $(F^\prime_1, G^\prime_1) = (F_1,G_1) + \delta (\lambda, \lambda).$ 
Let $\Phi_t:L[[t]] \to L[[t]]$ be the formal isomorphism obtained by extending linearly the map 
\[\Phi_t(a) ~=~ a + \lambda(a) t, ~~a \in L\] to $L[[t]].$
More precisely, $\Phi_t$ is an isomorphism of $K$-modules, where the inverse map is given by 
\[\Phi^{-1}_t(a) ~=~ a + \sum_{i \ge 1} (-1)^i \lambda^i(a) t^i.\]
We define $(F^\prime_t, G^\prime_t)$ as follows
\begin{eqnarray*}
F^\prime_t(a,b) &=& \Phi_t^{-1} (F_t(\Phi_t(a),\Phi_t(b))) \\
G^\prime_t (a,b,c) &=& \Phi_t^{-1}(G_t(\Phi_t(a),\Phi_t(b),\Phi_t(c)))
\end{eqnarray*}  
Note that $(F^\prime_t, G^\prime_t)$ is equivalent to $(F_t, G_t)$. Now we show that $(F^\prime_1, G^\prime_1)$ is the infinitesimal of $(F^\prime_t, G^\prime_t)$. Rewriting the expression of $F^\prime_t(a,b)$ we get 
\begin{eqnarray*}
\Phi_t^{-1}(F_t(\Phi_t(a),\Phi_t(b))) &=& \Phi_t^{-1} \left(\sum_{i \ge 0} F_i\big(a+\lambda(a)t, b+\lambda(b)t\big) ~t^i\right).
\end{eqnarray*}
Expanding the above equation we get the coefficient of $t$ as 
\[F_1(a,b) + [\lambda(a),b] + [a,\lambda(b)] - \lambda([a,b]) ~=~ F_1(a,b) + \delta_I \lambda(a,b) = F^\prime_1(a,b).\]
Again rewriting the expression of $G^\prime_t (a,b,c)$ we get 
\begin{eqnarray*}
\Phi_t^{-1}(G_t(\Phi_t(a),\Phi_t(b),\Phi_t(c))) &=& \Phi_t^{-1} \left(\sum_{i \ge 0} G_i\big(a+\lambda(a)t, b+\lambda(b)t, c+\lambda(c)t\big) ~t^i\right).
\end{eqnarray*}
Expand the above equation as before we get the coefficient of $t$ as 
\begin{multline*}
G_1(a,b,c) + \{\lambda(a),b,c\} + \{a,\lambda(b),c\} + \{a,b,\lambda(c)\} - \lambda(\{a,b,c\}) \\
=~ G_1(a,b,c) + \delta_{II} \lambda(a,b,c) ~=~ G^\prime_1 (a,b,c). 
\end{multline*}
Therefore $(F^\prime_1, G^\prime_1)$ is the infinitesimal of $(F^\prime_t, G^\prime_t)$. Hence, given any $(2,3)$-cochain $(F^\prime_1, G^\prime_1)$ cohomologous to $(F_1, G_1)$ there exists a one-parameter formal deformation $(F^\prime_t, G^\prime_t)$ of $L$ whose infinitesimal is $(F^\prime_1, G^\prime_1)$ and is equivalent to $(F_t, G_t).$  
\end{proof}

\begin{Rem}
It follows that the integrability of an element in $Z^{(2, 3)}(L; L)$ depends on its cohomology class. If any element representing a $(2, 3)$-cohomology class is integrable, that is, it is infinitesimal of a formal one-parameter deformation of $L,$ then every other element in the same cohomology class is also integrable.
\end{Rem}

The above discussions reveal that the infinitesimal of a trivial deformation is a coboundary, though the converse may not be true. In other words, a non-trivial deformation may have an infinitesimal which is a coboundary. However, what we are assure of is the following.

\begin{Thm}\label{non-trivial-deformation}
A non-trivial deformation of a Lie-Yamaguti algebra $(L, [~,~], \{~,~,~\})$ is equivalent to a deformation whose $n$-infinitesimal is not a coboundary for some $n \geq 1.$
\end{Thm}

\begin{proof}
Let $(F_t, G_t)$ be a formal one-parameter deformation of $L$ with its $n$-infinitesimal  $(F_n,G_n)$, $n \ge 1$. Then by Lemma \ref{n-infinitesimal-cocycle} we have $\delta(F_n,G_n) = (0,0).$ Now suppose $(F_n, G_n)$ is a coboundary, then $F_n = - \delta_I \lambda_n$ and $G_n = - \delta_{II}\lambda_n$ for some $\lambda_n \in C^1(L; L)$. Let $\Phi_t: L[[t]] \to L[[t]]$ be a formal isomorphism defined as 
\[\Phi_t(a) = a + \lambda_n(a) t^n.\]
We now define another deformation $(F^\prime_t, G^\prime_t)$  equivalent to $(F_t, G_t)$ as follows:
\begin{eqnarray*}
F^\prime_t(a,b) &=& \Phi_t^{-1}(F_t(\Phi_t(a),\Phi_t(b))) \quad=\quad \sum_{i \ge 0}F^\prime_i (a,b)t^i,\\
G^\prime_t (a,b,c) &=& \Phi_t^{-1}(G_t(\Phi_t(a),\Phi_t(b),\Phi_t(c))) \quad=\quad \sum_{i \ge 0} G^\prime_i (a,b,c) t^i.
\end{eqnarray*}
Explicitly, $F^\prime_t$ and $G^\prime_t$ are given by 
\begin{eqnarray*}
F^\prime_t (a,b) &=& [a,b] + \big([\lambda_n(a),b] + [a,\lambda_n(b)] - \lambda_n([a,b]) + F_n(a,b) \big) t^n \\
&~& \quad + F_{n+1}(a,b)t^{n+1} + \cdots \\
G^\prime_t (a,b,c) &=& \{a,b,c\} + \big( \{\lambda_n(a),b,c\} + \{a,\lambda_n(b),c\} + \{a,b,\lambda_n(c)\} - \lambda_n(\{a,b,c\})\\ 
&~& \quad  
+ G_n(a,b,c)\big)t^n + G_{n+1}(a,b,c)t^{n+1} + \cdots 
\end{eqnarray*}
Now suppose that $F_{n} \neq 0$. Then as $F_n = - \delta_I \lambda_n$, we have 
\[F_n(a,b) ~=~ - \big([\lambda_n(a),b] + [a,\lambda_n(b)] - \lambda_n([a,b]) \big)\]
Thus, the coefficient of $t^n$ in $F^\prime_t$ is zero. In case $F_n=0,$ then $\delta_I\lambda_n(a,b) = 0$ and hence coefficient of $t^n$ is again zero. Now by a similar argument the coefficient of $t^n$ in $G^\prime_t$ is also zero. Thus $(F^\prime_i, G^\prime_i)=(0,0)$ for all $1 \leq i \leq n$. Hence we can repeat the same argument to kill off an infinitesimal that is a coboundary and the process must stop if the deformation is non-trivial. This completes the proof of the theorem. 
\end{proof}

The notion of rigidity is defined as follows.

\begin{Def}
A Lie-Yamaguti algebra  $(L, [~,~], \{~,~,~\})$ is said to be rigid if every deformation of $L$ is a trivial deformation.
\end{Def}

As a consequence of Theorem \ref{non-trivial-deformation} we have the following sufficient condition for rigidity of a Lie-Yamaguti algebra.

\begin{Thm}
Let $(L,[~,~],\{~,~,~\})$ be a Lie-Yamaguti algebra with $H^{(2,3)}(L;L) = 0$, then $L$ is rigid.  
\end{Thm}

\begin{proof} 
Let $(L[[t]], F_t, G_t)$ be a formal one-parameter deformation of a Lie-Yamaguti algebra $L$ which is non-trivial. Then by Theorem \ref{non-trivial-deformation}  $( F_t, G_t)$ is equivalent to a deformation $(F^\prime_t, G^\prime_t)$ whose $n$-infinitesimal (which is a $(2, 3)$-cocycle by Proposition \ref{n-infinitesimal-cocycle}) is not a coboundary, for some $n \geq 1.$ This contradicts the fact that $H^{(2,3)}(L;L) = 0.$
\end{proof}

We end our discussion with an example of a rigid Lie-Yamaguti algebra in the next section. 

\vspace{0.1cm}
\section{\Large Free Object in the category of Lie-Yamaguti algebras}\label{$6$}
The aim of this last section is to construct the free object in the category of Lie-Yamaguti algebra and then prove that it is rigid.

Let $\mathbb K$ be a field and $B$ be a finite set. Let $L_0$ be a Lie-Yamaguti algebra over $\mathbb{K}$ with $i:B \to L_0$ a set map. 

\begin{Def}
	The Lie-Yamaguti algebra $L_0$ is called free over the set $B$ if for any Lie-Yamaguti algebra $L$ and any set map $f:B \to L$, there is a unique Lie-Yamaguti algebra morphism $\widetilde{f}:L_0 \to L$ such that the following diagram is commutative.
	\[\begin{tikzcd}
		B & {L_0} \\
		& L
		\arrow["f"', from=1-1, to=2-2]
		\arrow["{\widetilde{f}}", from=1-2, to=2-2]
		\arrow["i", from=1-1, to=1-2]
	\end{tikzcd}\]
\end{Def}
\begin{Thm}
	For any finite set $B$, there exists a free Lie-Yamaguti algebra $\mathcal{L}(B)$ over $B$, which is unique up to isomorphism.  
\end{Thm}

\begin{proof}
Since free Lie-Yamaguti algebra over $B$ is defined by the universal property it must be unique if exists. So it is enough to prove the existence of free Lie-Yamaguti algebra over $B$.  We proceed as follows.

We identify an element $b\in B$ with a planar $1$-ary tree with its leaf as $b,$ thus, identify $B$ with a set $\mathcal A$ of planar $1$-ary trees labelled by $B$ and call them {\it alphabets}. Therefore,
	$$\mathcal{A}:= \{\text{Set of all planar $1$-ary trees with their leaves as elements of $B$}\}.$$
So, the elements of $\mathcal{A}$ may be represented pictorially as follows
\newpage\noindent
{\bf Picture of $1$-ary tree with leaf $b$:}
	\begin{center}
		\begin{tikzpicture}
			\node (r) at (0,0) {};
			\node[circle,draw,fill=black,inner sep=1.5pt] (b) at (0,1) {};
			\node at (0,1.4) {b};
			
			\path[] (r) edge (b); 	
		\end{tikzpicture}
	\end{center}
where $b \in B$. Next, we consider two more sets:
\smallskip 
	$$\mathcal{B}:= \{\text{Set of all planar, binary, complete, rooted trees with leaves labelled by $B$}\},$$
	$$\mathcal{T}:= \{\text{Set of all planar, ternary, complete, rooted trees with leaves labelled by $B$}\}.$$

\medskip\noindent
For example, pictorially an element of $\mathcal B$ with two leaves $a, b \in B$ is represented as follows.\\
{\bf Picture of planar binary tree with leaves $a$ and $b$:}

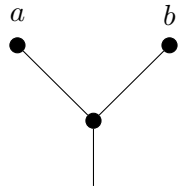
\begin{center}
		\begin{tikzpicture}
			\node (r) at (0,-1) {};
			\node[circle,draw,fill=black, inner sep=2pt] (ab) at (0,0) {};
			\node[circle,draw,fill=black, inner sep=2pt] (a) at (-1,1) {};
			\node at (-1,1.4) {$a$};
			\node[circle,draw,fill=black, inner sep=2pt] (b) at (1,1) {};
			\node at (1,1.4) {$b$};
			
			\path 
			(r) edge (ab)
			(ab) edge (a)
			(ab) edge (b)
			;
		\end{tikzpicture}
	\end{center} 

\noindent
Note that we may obtain all elements of $\mathcal B$ starting from the given alphabates in $\mathcal A$ by a method of {\it grafting} which we call type $I$ grafting, which alternatively, may be expressed as words by introducing a binary operation $``\cdot"$ on $B$ taking values in $\mathcal B.$ 

\smallskip
More explicitly, the above planar binary tree  is obtained by grafting the $1$-ary tree with leaf $a$ and the $1$-ary tree with leaf $b,$ in order, from left to right. Equivalently, we may represent this planar binary tree as a word by $a\cdot b.$ Similarly, for any elements $a,b,c \in B$, the words 
	$(a \cdot b) \cdot c$  and $a \cdot (b \cdot c)$ are represented by the following trees.\\
	
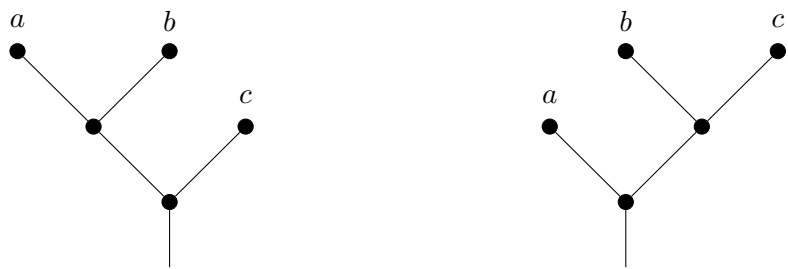
	\begin{center}
		\begin{tikzpicture}
			\begin{scope}
				\node[circle,draw,fill=black,inner sep=2pt] (a) at (-2,2) {};
				\node at (-2,2.4) {$a$};
				\node[circle,draw,fill=black,inner sep=2pt] (b) at (0,2) {};
				\node at (0,2.4) {$b$};
				\node[circle,draw,fill=black,inner sep=2pt] (ab) at (-1,1) {};
				\node[circle,draw,fill=black,inner sep=2pt] (c) at (1,1) {};
				\node at (1,1.4) {$c$};
				\node[circle,draw,fill=black,inner sep=2pt] (abc) at (0,0) {};
				\node (r) at (0,-1) {};
				
				\path[] (a) edge (ab)
				(b) edge (ab)
				(c) edge (abc)
				(ab) edge (abc)
				(abc) edge (r) 
				;  
			\end{scope}
			\begin{scope}[xshift=6cm]
				\node[circle,draw,fill=black,inner sep=2pt] (c) at (2,2) {};
				\node at (2,2.4) {$c$};
				\node[circle,draw,fill=black,inner sep=2pt] (b) at (0,2) {};
				\node at (0,2.4) {$b$};
				\node[circle,draw,fill=black,inner sep=2pt] (a) at (-1,1) {};
				\node at (-1,1.4) {$a$};
				\node[circle,draw,fill=black,inner sep=2pt] (bc) at (1,1) {};
				\node[circle,draw,fill=black,inner sep=2pt] (abc) at (0,0) {};
				\node (r) at (0,-1) {};
				
				\path[]
				(b) edge (bc)
				(c) edge (bc)
				(bc) edge (abc)
				(a) edge (abc)
				(abc) edge (r) 
				;  
			\end{scope}
		\end{tikzpicture}
	\end{center}
	
\noindent
Further, pictorially an element of $\mathcal T$ with leaves $a, b, c \in B$ is represented as follows.\\
{\bf Picture of planar ternary tree with leaves $a$, $b$ and $c$:}
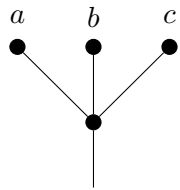
\begin{center}
\begin{tikzpicture}
\node (r) at (0,-1) {};
\node[circle,draw,fill=black, inner sep=2pt] (abc) at (0,0) {};
\node[circle,draw,fill=black, inner sep=2pt] (a) at (-1,1) {};
\node at (-1,1.4) {$a$};
\node[circle,draw,fill=black, inner sep=2pt] (b) at (0,1) {};
\node at (0,1.4) {$b$};
\node[circle,draw,fill=black, inner sep=2pt] (c) at (1,1) {};
\node at (1,1.4) {$c$};

\path[]
(r) edge (abc)
(a) edge (abc)
(b) edge (abc)
(c) edge (abc)
;
\end{tikzpicture}
\end{center}

As before, we may obtain all elements of $\mathcal T$ starting from the given alphabets in $\mathcal A$ by a method of {\it grafting} which we call type $II$ grafting, which alternatively, may be expressed as words by introducing a ternary operation  $``\star"$ on $B$ taking values in $\mathcal T.$ 

\smallskip
More explicitly, the above planar ternary tree is obtained by grafting the $1$-ary tree with leaf $a$, the $1$-ary tree with leaf $b,$ and the $1$-are tree with leaf $c,$ in order from left to right. Equivalently, we may represent this planar ternary tree as an word by $a\star b \star c.$

\noindent
Note that for any element $a,b,c \in B$, the word $a \star b \star c$ is represented by the following tree.

	
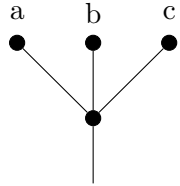
\begin{center}
		\begin{tikzpicture} 
			\node (r) at (0,-1) {};
			\node[circle,draw,fill=black, inner sep=2pt] (abc) at (0,0) {};
			\node[circle,draw,fill=black, inner sep=2pt] (a) at (-1,1) {};
			\node at (-1,1.4) {a};
			\node[circle,draw,fill=black, inner sep=2pt] (b) at (0,1) {};
			\node at (0,1.4) {b};
			\node[circle,draw,fill=black, inner sep=2pt] (c) at (1,1) {};
			\node at (1,1.4) {c};
			
			\path[]
			(r) edge (abc)
			(a) edge (abc)
			(b) edge (abc)
			(c) edge (abc)
			;
		\end{tikzpicture}
	\end{center}
	
Next, we expand the set $\mathcal{A} \cup \mathcal{B} \cup \mathcal{T}$ recursively, by repeated applications of type $I$ and type $II$ grafting and thereby, extend the binary and ternary operations to a bigger set as follows.

For any $\alpha, \beta \in \mathcal{A} \cup \mathcal{B} \cup \mathcal{T}$, denote $\alpha \cdot \beta$ by the tree
	
	
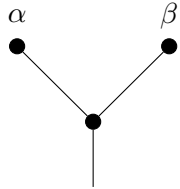
\begin{center}
		\begin{tikzpicture}
			\node (r) at (0,-1) {};
			\node[circle,draw,fill=black, inner sep=2pt] (ab) at (0,0) {};
			\node[circle,draw,fill=black, inner sep=2pt] (a) at (-1,1) {};
			\node at (-1,1.4) {$\alpha$};
			\node[circle,draw,fill=black, inner sep=2pt] (b) at (1,1) {};
			\node at (1,1.4) {$\beta$};
			
			\path 
			(r) edge (ab)
			(ab) edge (a)
			(ab) edge (b)
			;
		\end{tikzpicture}
	\end{center} 
Denote by $\mathcal{G}_1$ the set of all possible trees obtained by type $I$ grafting. For example, if $\alpha = a \star b \star c \in \mathcal{T},$ $ a, b, c \in \mathcal A$ and $\beta = d \in \mathcal A$ then the word $(a \star b \star c) \cdot d$ corresponds to the tree (given below) in $\mathcal{G}_1$. 
	
	
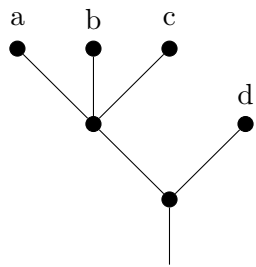
\begin{center}
		\begin{tikzpicture}
			\node (r) at (0,-1) {};
			\node[circle,draw,fill=black,inner sep=2pt] (alphabeta) at (0,0) {};
			\node[circle,draw,fill=black,inner sep=2pt]
			(alpha) at (-1,1) {};
			\node[circle,draw,fill=black,inner sep=2pt]
			(beta) at (1,1) {};
			\node at (1,1.4) {d};
			\node[circle,draw,fill=black,inner sep=2pt] 
			(a) at (-2,2) {};
			\node at (-2,2.4) {a};
			\node[circle,draw,fill=black,inner sep=2pt] 
			(b) at (-1,2) {};
			\node at (-1,2.4) {b};
			\node[circle,draw,fill=black,inner sep=2pt]
			(c) at (0,2) {};	
			\node at (0,2.4) {c};
			
			\path
			(r) edge (alphabeta)
			(alphabeta) edge (alpha)
			(alphabeta) edge (beta)
			(alpha) edge (a)
			(alpha) edge (b)
			(alpha) edge (c)
			;	
		\end{tikzpicture}
	\end{center}

\noindent
For any $\alpha, \beta, \gamma \in \mathcal{A} \cup \mathcal{B} \cup \mathcal{T}$, denote $\alpha \star \beta \star \gamma$ by the tree
\bigskip
	
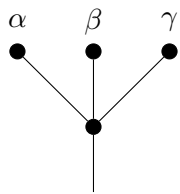
\begin{center}
		\begin{tikzpicture}
			\node (r) at (0,-1) {};
			\node[circle,draw,fill=black, inner sep=2pt] (abc) at (0,0) {};
			\node[circle,draw,fill=black, inner sep=2pt] (a) at (-1,1) {};
			\node at (-1,1.4) {$\alpha$};
			\node[circle,draw,fill=black, inner sep=2pt] (b) at (0,1) {};
			\node at (0,1.4) {$\beta$};
			\node[circle,draw,fill=black, inner sep=2pt] (c) at (1,1) {};
			\node at (1,1.4) {$\gamma$};
			
			\path[]
			(r) edge (abc)
			(a) edge (abc)
			(b) edge (abc)
			(c) edge (abc)
			;
		\end{tikzpicture}
	\end{center}
Denote by $\mathcal{G}_2$ the set of all possible trees obtained by type $II$ grafting. For example if $\alpha = a \in \mathcal{A}$, $\beta = b \in \mathcal{A}$, and $\gamma = c \cdot d \in \mathcal{B}$, then the word $a \star b \star (c \cdot d)$ corresponds to the following tree. 
	
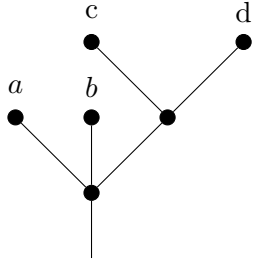
\begin{center}
		\begin{tikzpicture}
			\node (r) at (0,-1) {};
			\node[circle,draw,fill=black, inner sep=2pt] (abc) at (0,0) {};
			\node[circle,draw,fill=black, inner sep=2pt] (a) at (-1,1) {};
			\node at (-1,1.4) {$a$};
			\node[circle,draw,fill=black, inner sep=2pt] (b) at (0,1) {};
			\node at (0,1.4) {$b$};
			\node[circle,draw,fill=black, inner sep=2pt] (gamma) at (1,1) {};
			\node[circle,draw,fill=black, inner sep=2pt]
			(c) at (0,2) {};
			\node at (0,2.4) {c};
			\node[circle,draw,fill=black, inner sep=2pt]
			(d) at (2,2) {};
			\node at (2,2.4) {d};
			
			\path
			(r) edge (abc)
			(a) edge (abc)
			(b) edge (abc)
			(gamma) edge (abc)
			(gamma) edge (c)
			(gamma) edge (d)
			;
		\end{tikzpicture}
	\end{center}

Let us denote by $\mathcal W := \mathcal{G}_1 \cup \mathcal{G}_2$ the set of all possible words generated by $B$ as described above. Note that $\mathcal W$ is closed under the binary operation $``\cdot"$ and the ternary operation $``\star".$ Let  $\mathcal{W}(B)$ denote the $\mathbb{K}$-vector space generated by $\mathcal W.$  Extend the operations $``\cdot"$ and  $``\star"$ linearly over $\mathcal{W}(B).$ Then, $(\mathcal{W}(B), \cdot, \star)$ is the free non-commutative, non-associative algebra equipped with a binary and a ternary operations. 
	
\medskip
We call a subspace of $\mathcal{W}(B)$ an ideal in $\mathcal{W}(B)$ if it is closed under $``\cdot"$ and $``\star"$ by elements of $\mathcal{W}(B)$. Let $I$ be the ideal of $\mathcal{W}(B)$ generated by 
	\begin{eqnarray*}
		a \cdot a,\quad\quad\quad a \star a \star b,
	\end{eqnarray*}
	\begin{eqnarray*}
		\circlearrowleft_{(a,b,c)}(a\cdot b) \cdot c ~+ \circlearrowleft_{(a,b,c)} a \star b \star c, \quad\quad \circlearrowleft_{(a,b,c)} (a \cdot b) \star c \star d,
	\end{eqnarray*}
	\begin{eqnarray*}
		a \star b \star (c \cdot d) - (a \star b \star c) \cdot d - c \cdot (a \star b \star d), 
	\end{eqnarray*}
	\begin{eqnarray*}
		a \star b \star ( x \star y \star z) - (a \star b \star x) \star y \star z - x \star (a \star b \star y) \star z - x \star y \star (a \star b \star z)
	\end{eqnarray*}
for all $a,b,c,x,y,z \in \mathcal{W}(B).$ 

Consider $\mathcal{L}(B) = \mathcal{W}(B) / I.$ Then, it is clear that the vector space $\mathcal{L}(B)$ equipped with the binary operation $``\cdot"$  and the ternary operation  $``\star"$  induced from $\mathcal{W}(B)$ becomes a Lie-Yamaguti algebra over $\mathbb{K}$. Furthermore, it is easy to see that with the canonical map $i:B \to \mathcal{L}(B)$, $\mathcal{L}(B)$ is  a free Lie-Yamaguti algebra over the set $B$. 
\end{proof}

\begin{Def}
Let $V$ be a finite-dimensional vector space over $\mathbb K$ with basis a finite set $B.$ The free Lie-Yamaguti algebra over $V,$ denoted by $\mathcal L (V),$  is by definition 
$$ \mathcal{L}(V):= \mathcal{L}(B).$$
\end{Def}

\noindent
The final result of this paper is the following theorem.

\begin{Thm}
	The free Lie-Yamaguti algebra $\mathcal{L}(V)$ over a vector space $V$ is rigid. 
\end{Thm}

\begin{proof}
	Let us denote $\mathcal{L}(V)$ simply by $\mathcal{L}$. To show that $\mathcal{L}$ is rigid, it is enough to show $H^{(2,3)}(\mathcal{L};\mathcal{L}) = (0)$. Let $(f,g) \in Z^{(2,3)}(\mathcal{L};\mathcal{L})$ and $\bar{\mathcal{L}}$ denote the underlying vector space of $\mathcal{L}$. Now consider the short exact sequence of Lie-Yamaguti algebras
	\[\begin{tikzcd}
		0 & {\bar{\mathcal{L}}} & {\bar{\mathcal{L}} \oplus \bar{\mathcal{L}}} & {\mathcal{L}} & 0
		\arrow[from=1-1, to=1-2]
		\arrow["j", from=1-2, to=1-3]
		\arrow["\pi", from=1-3, to=1-4]
		\arrow[from=1-4, to=1-5]
	\end{tikzcd}\] 
	where $\bar{\mathcal{L}}$ is considered as an abelian Lie-Yamaguti algebra, and the Lie-Yamaguti algebra structure on $\bar{\mathcal{L}} \oplus \bar{\mathcal{L}}$ is defined as 
	\[[(a_1,b_1),(a_2,b_2)] := \big([a_1,b_2] + [b_1,a_2] - f(b_1,b_2),[b_1,b_2]\big)\]
	\[\{(a_1,b_1),(a_2,b_2),(a_3,b_3)\} := \big(\{a_1,b_2,b_3\}+\{b_1,a_2,b_3\}+\{b_1,b_2,a_3\}-g(b_1,b_2,b_3),\{b_1,b_2,b_3\}\big),\]
	$j$ being the inclusion into the first factor and $\pi$ the projection onto the second factor. This sequence splits as a sequence of vector spaces. So there exists a $\mathbb{K}$-linear map $\sigma:\mathcal{L} \to \bar{\mathcal{L}} \oplus \bar{\mathcal{L}}$ such that $\pi \circ \sigma= id_{\mathcal{L}}$. Hence $\sigma$ must be of the form $(\xi,id)$ where $\xi: \mathcal{L} \to \bar{\mathcal{L}}$ is a $\mathbb{K}$-linear map. Let $\sigma^\prime=\sigma|_V: V \to \bar{\mathcal{L}} \oplus \bar{\mathcal{L}}$. Since $\mathcal{L}$ is free over $V$ we get a Lie-Yamaguti algebra map $\widetilde{\sigma}: \mathcal{L} \to \bar{\mathcal{L}} \oplus \bar{\mathcal{L}}$ with $\tilde{\sigma} \circ i = \sigma^\prime$, with $i$ being the inclusion $V \hookrightarrow \mathcal{L}$. 
	\[\begin{tikzcd}
		V & {\mathcal{L}} \\
		& {\bar{\mathcal{L}} \oplus \bar{\mathcal{L}}}
		\arrow["{\sigma'}"', from=1-1, to=2-2]
		\arrow["i", from=1-1, to=1-2]
		\arrow["{\widetilde{\sigma}}", from=1-2, to=2-2]
	\end{tikzcd}\]
	This means that on $V,$ $\widetilde{\sigma}$ and $\sigma^\prime$ agree. Since $\pi$ and $\widetilde{\sigma}$  are both Lie-Yamaguti algebra maps we have $\pi \circ \widetilde{\sigma} = id$. Hence $\widetilde{\sigma}$ is of the form $(\lambda,id)$ for some $\mathbb{K}$-linear map $\lambda: \bar{\mathcal{L}} \to \bar{\mathcal{L}}$. Note that $\lambda \in C^1(\mathcal{L};\mathcal{L})$. Now as $\widetilde{\sigma}$ is a Lie-Yamaguti algebra morphism, we have 
	\[\widetilde{\sigma}[a,b] = [\widetilde{\sigma}(a),\widetilde{\sigma}(b)] \quad \text{and} \quad \widetilde{\sigma}\{a,b,c\} = \{\widetilde{\sigma}(a),\widetilde{\sigma}(b),\widetilde{\sigma}(c)\}\]
	Since $\widetilde{\sigma} = (\lambda,id)$ we get 
	\begin{eqnarray}
		\big(\lambda([a,b]),[a,b]\big)&=& [(\lambda(a),a),(\lambda(b),b)] \nonumber \\
		&=& \big([\lambda(a),b]+[a,\lambda(b_2)] - f(a,b),[a,b]\big) \label{e1}
		\\ \nonumber \\
		\big(\lambda(\{a,b,c\}),\{a,b,c\}\big) 
		&=& \{(\lambda(a),a),(\lambda(b),b),(\lambda(c),c)\} \nonumber\\
		&=&
		\big(\{\lambda(a),b,c\}+ \{a,\lambda(b),c\} + \{a,b,\lambda(c)\} - g(a,b,c),\{a,b,c\}\big) \nonumber \\
		& & \label{e2}
	\end{eqnarray}
	
	\vspace{-0.3cm}\noindent
	By the definition of coboundary maps we have 
	\begin{eqnarray*}
		\delta_I \lambda (a,b) &=& [\lambda(a),b]+[a,\lambda(b_2)] - f(a,b) \\ 
		\delta_{II} \lambda (a,b,c) &=& 
		\{\lambda(a),b,c\}+ \{a,\lambda(b),c\} + \{a,b,\lambda(c)\} - g(a,b,c)
	\end{eqnarray*}
	Equating the first coordinates of both sides of (\ref{e1}) and (\ref{e2}) we deduce, 
	\[(f,g) = \delta(\lambda,\lambda).\]
	So any $(2,3)$-cocycle in $C^{(2,3)}(\mathcal{L};\mathcal{L})$ is a $(2,3)$-coboundary. Hence $H^{(2,3)}(\mathcal{L};\mathcal{L}) = (0)$.
\end{proof}

\vspace{0.1in}
\noindent
{\bf Declarations}

\noindent\textbf{Conflict of interest} \\
The author declares that there is no conflict of interest. 

\noindent\textbf{Ethics approval} \\
This article does not contain any studies with human participants or animals performed by the authors. 

\noindent\textbf{Funding} \\
This research received no specific grant from any funding agency in the public, commercial, or not-for-profit sectors. 

\noindent\textbf{Data availability} \\
No data was used for the research described in the article.

\providecommand{\bysame}{\leavevmode\hbox to3em{\hrulefill}\thinspace}
\providecommand{\MR}{\relax\ifhmode\unskip\space\fi MR }
\providecommand{\MRhref}[2]{%
\href{http://www.ams.org/mathscinet-getitem?mr=#1}{#2}
}
\providecommand{\href}[2]{#2}

\end{document}